\documentclass[final]{siamltex}
\usepackage{amssymb,amstext,cite}
\usepackage{amsmath}
\usepackage{multirow}
\newcommand{\field}[1]{\mathbb{#1}}
\newcommand{\C}{\field{C}}
\newcommand{\R}{\field{R}}
\newtheorem{remark}[theorem]{Remark}
\allowdisplaybreaks
\numberwithin{equation}{section}

\usepackage{verbatim}

\begin{document}
\title{A proof of Einstein's effective viscosity for a dilute suspension of spheres\footnote{This paper was submitted for publication on October 1, 2010}}

\author{Brian M.~Haines\thanks{Corresponding author; Department of Mathematics, Penn State University, McAllister Bldg, University Park, PA 16802 ({\tt haines@math.psu.edu})} \and Anna L.~Mazzucato\thanks{Department of Mathematics, Penn State University, McAllister Bldg, University Park, PA 16802 ({\tt mazzucat@math.psu.edu})}}

\maketitle

\begin{abstract}
We present a mathematical proof of Einstein's formula for the effective viscosity of a dilute suspension of rigid neutrally--buoyant spheres when the spheres are centered on the vertices of a cubic lattice.  We keep the size of the container finite in the dilute limit and consider boundary effects.  Einstein's formula is recovered as a first-order asymptotic expansion of the effective viscosity  in the volume fraction.  To rigorously justify this expansion, we obtain an explicit upper and lower bound on the effective viscosity. A lower bound  is found using energy methods reminiscent of the work of Keller et al.  An upper bound follows by obtaining an explicit estimate for the tractions, the normal component of the stress on the fluid boundary, in terms of the velocity on the fluid boundary. This estimate, in turn, is established using a boundary integral formulation for the Stokes equation.  Our proof admits a generalization to other particle shapes and the inclusion of point forces to model self-propelled particles.
\end{abstract}



\pagestyle{myheadings}
\thispagestyle{plain}
\markboth{B.~M.~HAINES AND A.~L.~MAZZUCATO}{A PROOF OF EINSTEIN'S EFFECTIVE VISCOSITY}

\section{Introduction}

The purpose of this work is to give a rigorous justification of  Einstein's formula for the effective viscosity of a dilute suspension of rigid spherical particles in an ambient fluid.
Our approach differs from those previously employed in the literature and is well suited to study  the dilute limit for suspensions in a finite-size container


Einstein \cite{einstein1906} calculated the effective viscosity  $\hat{\eta}$ of a dilute suspension of spheres to be
\begin{equation}
\label{einstein_formula}
\hat{\eta} \approx\eta \left(1 +\frac{5}{2} \phi\right)
\end{equation}
where $\eta$ is the viscosity of the ambient fluid and $\phi$ is the volume fraction of spheres in the suspension. By a dilute suspension, we mean a suspension where the inter-particle distance $d$ is much larger than the particle size $a$. The formula above should be interpreted as a first-order calculation in an infinite asymptotic expansion of the true effective viscosity in powers of the volume fraction.

Einstein's result was extended to spheroids by Jeffery in \cite{jeffery}, who showed that the effective viscosity depends on the orientations of the spheroids. However,  since the distribution of orientations is not unique in the dilute limit, Jeffery obtained bounds, but not an asymptotic formula,  for the effective viscosity.  His work was completed by Hinch and Leal in \cite{leal+hinch,hinch+leal} who computed an ensemble--averaged effective viscosity by including the effects of rotational diffusion, which lead to a unique steady--state orientation distribution.  A general theory for dilute particle suspensions was later set out by Batchelor in \cite{batchelor69} and Brenner in \cite{brenner1972}.  In \cite{batchelor2spheres,batchelor1}, Batchelor and Green calculated  the $\mathcal{O}(\phi^2)$ correction to Einstein's result, by including the effects of pairwise interactions.  In \cite{almogbrenner}, Almog and Brenner produced a homogenized viscosity field $\hat{\eta}(x)$ that includes boundary effects and agrees with Einstein's result when the volume fraction is uniform.  In \cite{2d_ev,3d_ev_pre,3d_ev_CPAA}, the results for disks and spheroids were extended to dilute suspensions of active particles (particles that propel themselves, such as bacteria).

While most of the results in
\cite{einstein1906,jeffery,leal+hinch,hinch+leal,batchelor69,
brenner1972,batchelor2spheres,batchelor1,almogbrenner,
2d_ev,3d_ev_pre,3d_ev_CPAA} are based on formal asymptotic analysis, a  more rigorous justification of equation \eqref{einstein_formula} is given in \cite{keller1967}, using variational principles. In that work it is proved that for a suspension contained in a domain $\Omega$,
\begin{equation}
\label{keller_result}
\hat{\eta} =\eta \left(1 +\frac{5}{2} \phi + o(\phi)\right),
\end{equation}
as $\phi\to 0$, where $\phi$ is the volume fraction. They obtain this result by placing the particles in a finite-size container and letting the container expand in order to avoid the renormalization required in \cite{batchelor1} when considering a suspension filling $\mathbb{R}^3$.
The question of whether the same asymptotic expansion is valid in a container kept at finite volume when the size of the particles vanishes was not addressed.  Furthermore, it is not immediate that the method used in  \cite{keller1967} can be applied to the case of different particle shapes and self--propelled particles, as it requires knowledge of exact solutions to the Stokes equation for a particle inside a concentric sphere.

Concurrently, in \cite{ammari}, a proof of Einstein's result including a $\mathcal{O}(\phi^2)$ correction is given in the case of periodic boundary conditions.  The formula is obtained by considering a periodic suspension of a spherical fluid drops with viscosity $\eta^*$ in a fluid with viscosity $\eta$ in the limit $\eta^*\rightarrow \infty$.  In this paper, we consider a suspension in a finite-size container and derive the formula from a PDE model with rigid particles.

In this paper, we improve upon \eqref{keller_result} by showing that the next correction to Einstein's formula for a suspension in a finite-size container is at least of order $\phi^{3/2}$.
More precisely, we will show that
there exists a positive constant $C$ independent of the volume fraction $\phi$ and the number of particles $N$ such that for all $\phi
 < \frac{\pi}{6}$,
\begin{equation}
\label{goal_intro}
\left|\hat{\eta}-\eta \left(1 +\frac{5}{2} \phi\right)\right| \leq C \eta \phi^{3/2}.
\end{equation}
Unlike in previous work, we establish \eqref{goal_intro} assuming that the size of the container $|\Omega|$ remains finite in the dilute limit.
We model the suspended particles as spherical inclusions in $\Omega$ arranged in a regular array away from the walls of the container, though this arrangement is not crucial for our proof.  The condition  on $\phi$ in \eqref{goal_intro} above ensures that the spherical inclusions do not overlap.

The method we use differs from that in \cite{keller1967}
and admits a generalization to different particle shapes and the inclusion of particle self--propulsion, as modeled in \cite{2d_ev,3d_ev_pre,3d_ev_CPAA}.
We leave for further investigation to include collisions between particles in our model.
To derive a lower bound for the effective viscosity, we employ energy methods reminiscent of the work of Keller {\em et al} \cite{keller1967}.
To establish an upper bound, we utilize standard regularity estimates for solutions of the Stokes system, but we derive explicit values for all constants involved in these estimates.
A key step in the proof in fact consists in bounding explicitly the tractions at the fluid boundary in terms of the velocity on the fluid boundary, which in turn we achieve via boundary integral equations in terms of the Green's function for the Stokes problem.
These bounds have the same behavior in powers of $\phi$, namely $\phi^{3/2}$, thus implying \eqref{goal_intro}.

We should remark that
the cases of moderately ($d \approx a$) and densely ($d \ll a$) packed suspensions are not as well understood, though some rigorous mathematical results exist in both cases.  Moderately--packed suspensions are studied in \cite{nunan+keller}, where numerical results for spheres are presented, and in \cite{levy+sanchez}, where arbitrarily--shaped particles are studied using two--scale methods.  Densely--packed two--dimensional suspensions of disks are studied  using network approximations in \cite{berlyand-borcea-panchenko,berlyand-panchenko,berlyand+gorb+novikov}.


The paper is organized as follows. In Section \ref{sec_setup}, we introduce the PDE model for a suspension of spheres and define the effective viscosity.  Next, in Section \ref{sec_procedure}, we outline the procedure for proving equation \eqref{goal_intro}.  Section \ref{sec_tractions} contains the proof of a technical lemma on the continuity of solutions to a boundary integral equation related to Stokes equation.  Finally,  in Section \ref{sec_error_estimate}, the lemma  is applied to prove the main result of the paper, Theorem \ref{thm_ev_arb_n}.

We conclude the Introduction with some notational conventions.
$\Omega$  denotes a generic (bounded) domain in $\mathbb{R}^3$ with smooth boundary.
Throughout, we write $\| \cdot \|_{\Omega} := \| \cdot \|_{L^2(\Omega)}$ (or $\| \cdot \|_{(L^2(\Omega))^3}$ or $\| \cdot \|_{(L^2(\Omega))^{3\times 3}}$, as appropriate).
$H^\alpha(\Omega)$, $\alpha\in \mathbb{R}_+$,  denotes the standard $L^2$-based Sobolev space of functions on a smooth bounded domain, while the space $H^\alpha(\partial\Omega)$ of functions on the boundary is obtained via the usual trace relations.  In addition, we define   $C^\infty_0(\Omega):=\{f \in C^\infty(\Omega) | f (x)=0 \text{ for } x \in \partial \Omega \}$.
If $\sigma$ is a $2$ tensor and $n$ a vector, $\sigma \, n$ is their contraction:
\ $(\sigma n)^i  = \sum_{j} \sigma^i_j \, n^j$, while $:$ denotes the standard inner product between matrices, that is:
$A:B := \sum_{ij} A^i_j \, B_i^j$. \ Finally, we write $I$ for  the identity matrix.

\subsection{Preliminaries}
\label{sec_setup}

Let $\Omega\subset \R^3$ be a bounded domain with a $C^\infty$ boundary $\partial \Omega$.  We will model a
suspension of $N$ neutrally buoyant, passive spheres $B^l$ of radius $a$ centered at $x^l$ in an ambient fluid of viscosity $\eta$ using the Stokes
equation along with  balance of forces and torques at the boundary of the spheres:
\begin{equation}
\label{suspension_pde}
\left\{
\begin{array}{lr}
\eta \Delta u = \nabla p \text{, } \nabla\cdot u = 0, & x\in\Omega_F, \\
u = v^l+\omega^l\times (x-x^l), & x\in \partial B^l, \\
u = \epsilon x, & x \in \partial \Omega, \\
\int_{\partial B^l} \sigma n \, dx = 0 \text{, } \int_{\partial B^l} \sigma n \times (x-x^l) \, dx = 0 & l=1,\ldots,N. \\
\end{array}
\right.
\end{equation}
The use of the Stokes approximation to the full Navier-Stokes fluid equations is justified by the small scales and low velocities  (or equivalently small kinetic energy) inherent to the problem, so that the Reynolds number is typically very small.
Above, $u$ is the fluid velocity, $p$ is the pressure, $\Omega_F:=\Omega \setminus \cup_l B^l$ is the fluid domain,
$v^l\in\R^3$ and $\omega^l\in\R^3$ are the rigid translational and rotational velocities of the $l$th sphere respectively,
$\sigma := -p I+ 2 \eta e$  is the stress (here $I$ is the identity matrix) and $e:=\frac{1}{2}\left(\nabla u + \nabla u^T\right)$ is the strain or symmetric deformation gradient.
On the boundary of the domain $\partial\Omega$, we impose a linear profile for the flow $ u = \epsilon x$, where  $\epsilon$ is
a constant, symmetric, and traceless $3\times 3$ matrix.  The existence of solutions $(u,p)\in (C^\infty(\overline{\Omega}_F))^{4}$
is a standard result (see, for example, Theorems IV.1.1 and IV.5.2 in \cite{galdi}).  As in \cite{einstein1906}, we will study the case when $a \ll d$, where $d$ is
the minimum distance between spheres.  To simplify some calculations, we will consider the case where
$\Omega=B(0,R)$ and the spheres $B^l$ are arranged such that their centers lie
on the points of the array $\mathbb{Z}^3\cap B(0,R-1)$  and, in particular, then $d=1$.

\begin{remark}
This geometric set up is not crucial to our method. In particular, our estimates are valid for arbitrary $N$ and $d$ as $a\to 0$, so that we can assume that $|\Omega| \propto N\, d^3$ is  fixed throughout.
\end{remark}

We next give a definition of effective viscosity in terms of the kinetic energy in the system.
We denote with
\begin{equation}
\label{def_energy}
E(u) := 2 \eta \int_{\Omega_F} e: e \, dx.
\end{equation}
the rate of energy dissipation in $\Omega_F$.
Let $u^h=\epsilon x$.
Then, we define the effective viscosity ${\hat{\eta}}$ as the uniform viscosity in $\Omega$ such that the linear flow
$u^h = \epsilon x$ dissipates as much kinetic energy \ $E(u^h) := 2 \hat{\eta} |\Omega | \epsilon : \epsilon$
as the flow $u$ in the suspension, that is:
\begin{equation}
\label{definition_of_ev}
\hat{\eta} := \frac{1}{2} \frac{E(u)}{|\Omega | \epsilon : \epsilon},
\end{equation}
where $|\Omega |$ is the volume of $\Omega$.
Therefore, $\Hat{\eta}$ is a function of the fluid flow $u$, although we do not explicitly show this dependence.

\subsection{Outline}
\label{sec_procedure}

We briefly outline the key steps in proving our main result,  the validity of estimate \eqref{goal_intro}  for $\phi:=\frac{4\pi}{3}(a/d)^3 <\frac{\pi}{6}$ (see Theorem \ref{thm_ev_arb_n}).
This condition on the volume fraction ensures that the particles in suspension, modeled as spherical inclusions, do not overlap (recall $d=1$).

It was first shown by Einstein in \cite{einstein1906} that the first-order correction   $\frac{5}{2} \eta\,\phi$ to the ambient viscosity can be formally obtained by replacing $u$ in equation \eqref{definition_of_ev} with a dilute approximation $u^d$ (which will be defined explicitly in equation \eqref{new_u^d}) and $\sigma$ with the corresponding stress $\sigma^d$.  The dilute approximation $u^d$ is given by superimposing the Stokes solutions in the case of a single inclusion, one for each particle in the suspension. Interactions between particles are completely neglected in this approximation.
Our first step is to show in Lemma \ref{lower_bound_lemma} that, using the dilute solution $u^d$ instead of $u$ in equation \eqref{definition_of_ev} leads to the following bound on the effective viscosity $\Hat{\eta}$:
\[
\eta \left(1+\frac{5}{2}\phi\right) \leq \hat{\eta}+ C\eta \phi^{5/3}
\]
for a positive constant $C$ independent of $\phi$ and $N$.
Next, in Lemma \ref{upper_bound_lemma}, we prove an upper bound with the same order in $\phi$ by calculating the effective viscosity for  $u^+$ (denoted $\hat{\eta}(u^+)$), where $u^+$ solves
\begin{equation*}
\left\{
\begin{array}{lr}
\eta \Delta u^+ = \nabla p^+ \text{, } \nabla \cdot u^+ = 0& x \in \Omega_F \\
u^+ = \epsilon x^l & x \in \partial B^l \\
u^+ = \epsilon x & x \in \partial \Omega.
\end{array}
\right.
\end{equation*}
Even if we do not have an explicit form for $u^+$ at our disposal,  in Lemma \ref{lemma_ev_representation_arb_n} we obtain the following estimate on $\hat{\eta}(u^+)$:
\begin{equation*}
\left|\hat{\eta}(u^+)-\eta \left(1+ \frac{5}{2} \phi\right)\right| \leq \left|
\frac{1}{2\epsilon : \epsilon |\Omega |} \sum_l \epsilon:\int_{\partial B^l} \sigma^E n (x - x^l) dS \right| + C \eta \phi^{5/3},
\end{equation*}
by utilizing the partial differential equation (PDE for short) that $u^E:= u^+ - u^d$ satisfies (see Equation \eqref{error_pde_arb_n}).
Above, $\sigma^E$ is the stress corresponding to the flow $u^E$, $n$ the unit outer normal to the boundary, and $\sigma^E n$ are the corresponding tractions.
Since $|\Omega | \propto N$, we have
\begin{align}
\left| \hat{\eta}(u^+) - \eta\left(1+  \frac{5}{2} \phi \right) \right|
\leq
& C \frac{a^2}{N} \left\| \sigma^E  n \right\|_{\cup_l \partial B^l}+ C \eta \phi^{5/3}. \nonumber
\end{align}
The last step consists in estimating the tractions, which we do not know explicitly.
To this end, in  Theorem \ref{continuity_thm} we show that
$$
\left\| \sigma^E n \right\|_{\partial \Omega_F} \leq C \left\| u^E \right\|_{H^1(\partial \Omega_F)}.
$$
for  $C$ a positive constant  depending on the domain $\Omega_F$.
The most laborious part of the paper concerns  determining an explicit upper bound for $C$. which entails bounding the pressure at the surface of the spheres in terms of the velocity.  The procedure for doing so is outlined in Section \ref{sec_tractions} and the actual bound is given in Corollary \ref{final_t_estimate}.  Finally,   $\left\| u^E \right\|_{H^1(\partial \Omega_F)}$ can be calculated from the PDE it satisfies.  This estimate and the proof of the validity of equation \eqref{goal_intro} (Theorem \ref{thm_ev_arb_n}) are given in Section \ref{sec_error_estimate}.

\section{Estimating tractions on the boundary}
\label{sec_tractions}

As outlined above in Section \ref{sec_procedure}, to prove \eqref{goal_intro} we shall need to estimate the tractions $t=\sigma n$
on the fluid boundary $\partial \Omega_F$ in terms of the boundary velocity values.
These can be related via the following boundary integral equation: for $\xi\in \partial \Omega_F$ (see, for example, \cite{pozrikidis}),
\begin{equation}
\label{bie}
\frac{1}{2} u(\xi) =-\int_{\partial\Omega_F} \left[ \mathcal{G} (\xi-\nu) \sigma(\nu) n(\nu) - n(\nu) \Sigma(\xi,\nu) u(\nu) \right]dS_\nu,
\end{equation}
where
$
\mathcal{G}=(8\pi\eta)^{-1}\left(
I/r+ x x/r^3\right),
$
is the Oseen tensor, $\mathcal{P}=(4\pi)^{-1}x/r^3$ is the corresponding pressure, and
$\Sigma_{ijk} = -\mathcal{P}_j \delta_{ik} +
\eta \left(\frac{\partial\mathcal{G}_{ij}}{\partial x_k}+\frac{\partial\mathcal{G}_{kj}}{\partial x_i}\right)
$
is the corresponding stress.
The solvability of \eqref{bie} for the tractions in Sobolev spaces is known.

\begin{theorem}
\label{continuity_thm}
Let $\Omega \subset \R^3$ be a bounded domain with a smooth boundary $\Gamma=\partial \Omega$.
Let $u$ be the solution of the Stokes equation in $\Omega$ with  Dirichlet boundary data $g\in (H^\alpha(\Gamma))^3$, $\alpha \geq \frac{1}{2}$, satisfying the compatibility condition
$\int_{\Gamma} g \cdot n \, dx = 0$. Then , there exist tractions $t:=\sigma n\in (H^{\alpha-1}(\Gamma))^3$  such that
$$
\| t \|_{(H^{\alpha-1}(\Gamma))^3} \leq C \| g \|_{(H^\alpha(\Gamma))^3},
$$
where $C>0$ is a constant depending on $\Omega$, but independent of $g$ and $t$.
\end{theorem}

We omit the proof of this Theorem.

In the remainder of this section, we will be concerned with finding an explicit upper bound for the constant $C$ in Theorem \ref{continuity_thm} above for the case $\alpha = 1$.

To this end, we will combine four inequalities, for which we evaluate constants explicitly.  First, in Section \ref{t_est_section}  extending the divergence-free condition  to the boundary for regular $u$ and a density argument gives
\begin{equation}
\label{t_estimate}
\| t \|_{\partial \Omega_F}
\leq \| p \|_{\partial \Omega_F} +  \eta \left\| \frac{\partial u}{\partial n} \right\|_{\partial \Omega_F} +
\sqrt{2}\eta \sum_{i=1}^2 \left\| \frac{\partial u}{\partial \tau^i} \right\|_{\partial \Omega_F},
\end{equation}
where $\tau^i$, $i=1,2$,  are orthogonal unit tangent vectors to the surface $\partial \Omega_F$.
Next, we estimate the pressure and the normal derivative appearing on the right side of
\eqref{t_estimate}.
In Section \ref{du_dn_estimate_section} we then show that
\begin{equation}
\label{dudn_estimate_preliminary}
C_1 \left\| \frac{\partial u}{\partial n} \right\|_{\partial\Omega_F}
\leq C_2 \sum_{i=1}^2 \left\| \frac{\partial u}{\partial \tau^i} \right\|_{\partial\Omega_F}
+ \frac{C_3}{\eta} \| p \|_{\partial \Omega_F}
+ C_4 \| u \|_{\partial \Omega_F}
+ \frac{C_5}{\eta} \|p \|_{\Omega_F},
\end{equation}
where
\begin{align}
C_1 &= \sqrt{2}/2- \sqrt{(d/2-a)^{-1}\left(1+(4\delta_2)^{-1} \right)\delta_4}, \label{C_1} \\
C_2 &= \sqrt{2}/2\left( 1+ \sqrt{\delta_1^{-1}} \right), \label{C_2} \\
C_3 &= \sqrt{\delta_1} + \sqrt{(d/2-a)^{-1} \left(1+(4\delta_2)^{-1}\right)\delta_3}, \label{C_3} \\
C_4 &=\sqrt{(d/2-a)^{-1}} \left[ \sqrt{1+(4\delta_2)^{-1} } \left(\sqrt{(4\delta_3)^{-1}} + \sqrt{(4\delta_4)^{-1}} \right) \right], \label{C_4} \\
C_5 &= \sqrt{\delta_2/(d/2-a)}, \label{C_5}
\end{align}
for all $\delta_1, \delta_2, \delta_3, \delta_4 >0$.
This estimate follows from standard integration by parts once an appropriate  extension of the normal vector $n$ to $\partial\Omega_F$  to a function on  $\overline{\Omega}_F$ is constructed.
In Section \ref{p_interior_estimate_section} we next show  that there exists $C>0$ such that
\begin{equation}
\label{p_inequality_preliminary}
\| p + C \|_{\Omega_F} \leq (1+\sqrt{2}) \beta \|p +C\|_{\partial \Omega_F},
\end{equation}
where
\begin{equation}
\label{beta}
\begin{split}
\beta =&
2^{1/8}\sqrt{1+6/d}\left(d/2
+a^{-1}+\sqrt{2} \right)^{1/4}
\\&
+2^{-1/4}\sqrt{1+3(R-d/2)^{-1}} \left( 1+R/3 \right)^{1/4}.
\end{split}
\end{equation}
This estimate is obtained construction an explicit extension $\tilde{p}$ of the boundary trace of $p$ on $\partial \Omega_F$ to  a neighborhood of each $\partial B^l$ and $\partial \Omega$ to obtain a satisfactory dependence of $\beta$ on $\Omega_F$.
Finally, we show in Section \ref{p_bdy_estimate_section} that
\begin{equation}
\label{p_bdy_estimate_preliminary}
\| p \|_{\partial \Omega_F} \leq  \frac{1}{2} \| t \|_{\partial \Omega_F} + \frac{3 \eta}{a} \| u \|_{\cup_l \partial B^l}
+ \frac{3 \eta}{R} \| u \|_{\partial B(0,R)}+ \eta \left\| \frac{\partial u}{\partial n} \right\|_{\partial \Omega_F}.
\end{equation}
This estimate is derived from a boundary integral equation for the pressure on the fluid boundary.  The norm of the associated integral operator can be estimated via  the Fourier transform by stereographically projecting the sphere onto a plane. Some of the calculations are simplified utilizing the quaternion formalism and  the Clifford algebra $\mathcal{A}_3$.

Combining inequalities \eqref{t_estimate}, \eqref{dudn_estimate_preliminary}, \eqref{p_inequality_preliminary}, and \eqref{p_bdy_estimate_preliminary}, we obtain
(we recall that $p$ is defined up to an arbitrary additive constant):

\begin{theorem}
\label{final_estimate_theorem}
Let $0<2a<d<R\in\R$ and let $u,p$ solve
\begin{equation*}
\left\{
\begin{array}{lr}
\eta \Delta u = \nabla p \text{, } \nabla \cdot u = 0 & x \in \Omega_F \\
u = g \in H^1\cap L^\infty(\partial \Omega_F) & x \in \partial \Omega_F,
\end{array}
\right.
\end{equation*}
where $\int_{\partial \Omega_F} g \cdot n dS = 0$, $\Omega_F:= B(0,R) \setminus \Omega_B$, $\Omega_B = \bigcup_{l\in \Lambda} B(x^l, a)$, $x^l \in d\mathbb{Z}^3$,
and
$
\Lambda=\left\{ l \in \mathbb{Z}^3 \big| \text{ } |x^l| + \frac{d}{2}-a < R \right\}.
$
Then there exists a normalization of $p$ such that
\begin{equation} \label{final_estimate_traction}
\begin{split}
\| t \|_{\partial \Omega_F}
\leq
\eta \frac{10}{7} \left[\frac{6}{R}+3\frac{C_4}{2\sqrt{2}} \right] \| u\|_{\partial \Omega_F}
+ \eta \frac{60}{7} \frac{1}{a} \| u \|_{\cup_l \partial B^l} \\
 +\eta \frac{10}{7} \left[\frac{3}{4}+25\sqrt{2} \right] \sum_{i=1}^2 \left\| \frac{\partial u}{\partial \tau^i} \right\|_{\partial \Omega_F},
\end{split}
\end{equation}
where
$C_4$ is given in equation \eqref{C_4} 
and $\delta_k$ are chosen as in equation \eqref{deltas}.

\end{theorem}

\begin{proof}
Combining inequalities \eqref{t_estimate}, \eqref{dudn_estimate_preliminary}, and \eqref{p_inequality_preliminary} and renormalizing $p$ such that
\begin{equation}
\label{renormalized_p_ineq}
\| p \|_{\Omega_F} \leq (1+\sqrt{2}) \beta \|p \|_{\partial \Omega_F}
\end{equation}
yields
\begin{equation}
\label{C_estimate_thm_eq_1}
\begin{split}
\| t \|_{\partial \Omega_F} \leq & \left[ 1 +  C_1^{-1}\left(C_3+ (1+\sqrt{2}) C_5 \beta \right) \right] \| p \|_{\partial \Omega_F} \\
& + \eta \left( \sqrt{2} + \frac{C_2}{C_1} \right) \sum_{i=1}^2 \left\| \frac{\partial u}{\partial \tau^i} \right\|_{\partial \Omega_F}
+ \eta \frac{C_4}{C_1} \| u \|_{\partial \Omega_F}.
\end{split}
\end{equation}
Combining inequalities \eqref{p_bdy_estimate_preliminary}, \eqref{dudn_estimate_preliminary}, and \eqref{renormalized_p_ineq} produces
\begin{equation}
\label{C_estimate_thm_eq_2}
\begin{split}
\| p \|_{\partial \Omega_F} & \left[1 - C_1^{-1} \left(C_3+ (1+\sqrt{2}) C_5 \beta \right) \right] \\
\leq & \frac{1}{2} \| t \|_{\partial \Omega_F} + \eta \left[ \frac{3}{R}+\frac{C_4}{C_1} \right] \| u \|_{\partial \Omega_F}
+ \eta \frac{3}{a} \| u \|_{\cup_l \partial B^l}
+ \eta \frac{C_2}{C_1} \sum_{i=1}^2 \left\| \frac{\partial u}{\partial \tau^i} \right\|_{\partial \Omega_F}.
\end{split}
\end{equation}
We wish to use inequality \eqref{C_estimate_thm_eq_2} to remove $\| p \|_{\partial \Omega_F}$ from equation \eqref{C_estimate_thm_eq_1}.  In order to do this, we must choose
$\delta_k$ (in $C_l$) such that
\begin{equation*}
C_6 =
\left[1 - C_1^{-1} \left(C_3+ (1+\sqrt{2}) C_5 \beta  \right) \right]^{-1}
\left[ 1 +  C_1^{-1} \left(C_3 + (1+\sqrt{2}) C_5 \beta  \right) \right] < 2,
\end{equation*}
which, in turn, requires that $C_1^{-1} \left(C_3 + (1+\sqrt{2}) C_5 \beta \right) < 1/3$.
This can be achieved by choosing
\begin{equation}
\label{deltas}
\begin{array}{cccc}
\delta_1 = \frac{1}{2^{11}}, & 
\delta_2 = \frac{1}{2^{10}} \frac{d-2a}{(1+\sqrt{2})^2\beta^2}, & 
\delta_3 = \frac{1}{2^{10}} \frac{\delta_2 (d-2a)}{4\delta_2+1}, & 
\delta_4 = 2^8 \delta_3. 
\end{array}
\end{equation}
Then $C_1=2\sqrt{2}$, $C_2=32+ 1/\sqrt{2}$, $C_6=3/5$, and $\eta \frac{1}{C_1} \left(C_3 + (1+\sqrt{2}) C_5 \beta  \right)=\frac{1}{6}$.
Combining equations \eqref{C_estimate_thm_eq_1}--\eqref{deltas}, we get the desired result.

\end{proof}

To apply this to the error estimate, we will use
\begin{corollary}
\label{final_t_estimate}
Let $u,p$ solve equation \eqref{suspension_pde}.
Assume, without loss of generality, that $d=1$ (and hence $R \propto N^{1/3}$).  Then $\exists C>0$, independent of $a$ and $N$ such that
\begin{equation*}
\begin{split}
\frac{\| t \|_{\partial \Omega_F}}{C \eta}
\leq &  \frac{N^{1/2}}{a^{1/2}}  \| u \|_{\partial \Omega_F} +  \frac{1}{a} \| u \|_{\cup_l \partial B^l}+  \sum_{i=1}^2 \left\| \frac{\partial u}{\partial \tau^i} \right\|_{\partial \Omega_F}.
\end{split}
\end{equation*}
\end{corollary}

\subsection{Estimate of $\| t \|_{\partial \Omega_F}$ }
\label{t_est_section}

In this section we derive the first of four inequalities, estimate \eqref{t_estimate} on the tractions at the boundary of $\Omega_F$. In what follows, we fix an orthonormal frame
$\{\tau_1,\tau_2,n\}$ on $\partial \Omega_f$, where $n$ is the unit outer normal and $\tau^i$ are tangent to the boundary.

\begin{lemma}
\label{incompressibility_at_bdy}
Let $u,p$ solve
\begin{equation}
\label{stokes_dirichlet}
\left\{
\begin{array}{lr}
\eta \Delta u = \nabla p \text{, } \nabla \cdot u = 0 & x \in \Omega_F \\
u = g \in H^1(\partial \Omega_F) & x \in \partial \Omega_F,
\end{array}
\right.
\end{equation}
where
$\int_{\partial \Omega_F} g \cdot n \, dS = 0$.  Then
\begin{equation}
\label{incompressibility_on_bdy}
\int_{\partial \Omega_F} \left| \frac{\partial u}{\partial n} \cdot n \right|^2 dS \leq 2 \sum_{i=1}^2 \int_{\partial \Omega_F} \left| \frac{\partial u}{\partial \tau^i} \cdot \tau^i \right|^2 dS.
\end{equation}
\end{lemma}


\begin{proof}
The Lemma is an immediate consequence of the divergence-free conditions for smooth solutions to \eqref{stokes_dirichlet} and follows by density for $u \in (H^{3/2}(\Omega_F))^3$.
Let $g_m\in (C^\infty(\partial \Omega_F))^3$ be a sequence of functions satisfying  \
$\int_{\partial \Omega_F} g_m \cdot n \, dS = 0$ that converges to $g $  in $(H^1(\partial  \Omega_F))^3$. Let $u_m, p_m$ solve equation \eqref{stokes_dirichlet} with  data $g_m$.
By Theorem IV.5.2 in \cite{galdi}, $u_m \in (C^\infty(\bar{\Omega}_F))^3$,  and therefore  $\nabla \cdot u_m = 0$ on $\partial \Omega_F$ by continuity. Consequently,
\begin{equation}
\label{incompressibility_on_bdy_approximation}
\int_{\partial \Omega_F} \left| \frac{\partial u_m}{\partial n} \cdot n \right|^2 dS \leq 2 \sum_{i=1}^2 \int_{\partial \Omega_F} \left| \frac{\partial u_m}{\partial \tau^i} \cdot \tau^i \right|^2 dS.
\end{equation}
Then,  regularity results for the Stokes equation (see Theorem IV.6.1 in \cite{galdi})  and the trace theorem give
\begin{equation*}
\begin{split}
\| \nabla u - \nabla u_m \|_{\partial \Omega_F}
&\leq C_1 \| u - u_m \|_{(H^{3/2}(\Omega_F))^3}
\leq C_2 \| g-g_m \|_{(H^1(\partial\Omega_F))^3}.
\end{split}
\end{equation*}
That is,  both the normal and tangential derivatives of $u_m$ at the boundary converge in $L^2$ to those of $u$.
We can hence conclude by passing to  the limit $m \rightarrow \infty$ in \eqref{incompressibility_on_bdy_approximation}.
\end{proof}

\begin{theorem}
Let $(u,p)$ solve equation \eqref{stokes_dirichlet}.  Then,
\begin{equation*}
\| t \|_{\partial \Omega_F}
\leq \| p \|_{\partial \Omega_F} +  \eta \left\| \frac{\partial u}{\partial n} \right\|_{\partial \Omega_F} +
\sqrt{2}\eta \sum_i \left\| \frac{\partial u}{\partial \tau^i} \right\|_{\partial \Omega_F}.
\end{equation*}

\end{theorem}

\begin{proof}
Applying Lemma \ref{incompressibility_at_bdy}, we have
\begin{equation*}
\begin{split}
\| t \| \leq &\| p n \| + \eta \| \nabla u \cdot n + \nabla u^T \cdot n \|
\leq \| p \| +  \eta \left\| \frac{\partial u}{\partial n} \right\|
+ \eta\left\|\frac{\partial u}{\partial n}\cdot n + \sum_i \frac{\partial u}{\partial \tau^i}\cdot n
\right\| \\
\leq & \| p \| +  \eta \left\| \frac{\partial u}{\partial n} \right\| +
\sqrt{2}\eta \sum_i \left\| \frac{\partial u}{\partial \tau^i} \right\|. \\
\end{split}
\end{equation*}
\end{proof}

\subsection{Estimate of $\left\| \frac{\partial u}{\partial n} \right\|_{\partial \Omega_F}$}
\label{du_dn_estimate_section}

In this section, we tackle the second inequality \eqref{dudn_estimate_preliminary}.
In order to bound $\left\| \frac{\partial u}{\partial n} \right\|_{\partial \Omega_F}$, we use
the following result.

\begin{theorem}
Let $0<2a<d<R\in \R$ and $\Omega_F:= B(0,R) \setminus \Omega_B$, where $\Omega_B = \bigcup_{l\in \Lambda} B(x^l, a)$, $x^l \in d\mathbb{Z}^3$,
and $\Lambda=\left\{ l \in \mathbb{Z}^3 \big| \text{ } |x^l| + \frac{d}{2}-a < R \right\}$.
Let $(u,p)$ solve  \eqref{stokes_dirichlet} with $g\in H^1(\partial \Omega_F)$.
Then, for all $\delta : =(\delta_1,\delta_2,\delta_3,\delta_4) > 0$,
\begin{equation}
\label{dudn_estimate}
C_1 \,  \left\| \frac{\partial u}{\partial n} \right\|_{\partial\Omega_F}
\leq C_2 \,\sum_{i=1}^2 \left\| \frac{\partial u}{\partial \tau^i} \right\|_{\partial\Omega_F}
+ \frac{C_3} {\eta} \| p \|_{\partial \Omega_F}
+ C_4\, \| u \|_{\partial \Omega_F}
+ \frac{C_5} {\eta} \|p \|_{\Omega_F},
\end{equation}
where $C_1(\delta)$--$C_5(\delta)$ are given in equations (\ref{C_1}--\ref{C_5}).

\end{theorem}

\begin{proof}
We again smooth out the data $g$ first and use density arguments.
Let $g^m \in (C^\infty (\partial \Omega_F))^3$ be such that $\int_{\partial \Omega_F} g^m \cdot n \, dS = 0$ and $g^m \rightarrow g$ in $(H^1(\partial \Omega_F))^3$.
Let $(u_m,p_m)$ be the solution to \eqref{stokes_dirichlet} with boundary data $g_m$.
Then, $u^m \in (C^\infty (\overline{\Omega}_F))^3$, $p^m \in C^\infty(\overline{\Omega}_F)$, and $u^m \rightarrow u$ in $(H^{3/2}(\Omega_F))^3$, $p^m \rightarrow p$ in $H^{1/2}(\Omega_F)$  (see, for example, Theorem IV.5.2 \cite{galdi}).

We next construct an extension of $n\in (C(\partial\Omega_F))^3$ to a function $\mathcal{N}\in (C(\Omega_F))^3\cap (L^2(\Omega_F))^3$ by extending it radially in the vicinity of all spherical inclusions and the spherical boundary.  Specifically, let $\mathcal{N}\in (C(\Omega_F))^3\cap (L^2(\Omega_F))^3$ be such that $\mathcal{N}=n$ on $\partial \Omega_F$.  Let $\mathcal{N}=0$ on $\partial B\left(x^l,\frac{d}{2}\right)$ for $l\in \Lambda$ and extend $\mathcal{N}$ inside $B\left(x^l,\frac{d}{2}\right) \setminus B(x^l,a)$ such that each component is linear on any line emanating from $x^l$.  Similarly, let $\mathcal{N}=0$ on $\partial B\left(0,R-\frac{d}{2}\right)$ and extend $\mathcal{N}$ inside $B(0,R)\setminus B\left(0,R-\frac{d}{2}\right)$ such that each component is linear on any line emanating from the origin.  Where it has not yet been defined, set $\mathcal{N}=0$.  It is easy to check that $\mathcal{N}$ has weak derivatives in $L^2(\Omega_F)\cap L^\infty(\Omega_F)$, that
\begin{equation}
\label{nabla_N_L_infty}
\begin{array}{cc}
\|\mathcal{N} \|_{(L^\infty(\Omega_F))^3} \leq 1, &
\left\| \nabla \mathcal{N} \right\|_{(L^\infty(\Omega_F))^{3 \times 3}} \leq \frac{1}{d/2-a},
\end{array}
\end{equation}
and $\nabla \cdot \mathcal{N} \geq 0$ almost everywhere.

By the definition of $\mathcal{N}$, we can write
\begin{equation*}
\begin{split}
\int_{\partial \Omega_F} \left| \frac{\partial u^m}{\partial n} \right|^2 dS
=& \int_{\partial \Omega_F} \left( \frac{\partial u^m}{\partial x} \cdot n \right)\cdot \frac{\partial u^m}{\partial n} dS
= \int_{\partial \Omega_F} \left( \frac{\partial u^m}{\partial x} \cdot \mathcal{N} \right) \cdot \frac{\partial u^m}{\partial n} dS \\
=& \int_{\Omega_F} \nabla \cdot \left[ \left(\nabla u^m  \mathcal{N}\right)  \nabla u^m \right] \, dx.
\end{split}
\end{equation*}
Performing the differentiation, using the Stokes equation, and integrating by parts, we get that
\begin{equation*}
\begin{split}
\int_{\partial \Omega_F} \left| \frac{\partial u^m}{\partial n} \right|^2 dS
=& \int_{\Omega_F} \left[ \nabla^2 u^m \mathcal{N} : \nabla u^m
+ \nabla u^m \nabla\mathcal{N} : \nabla u^m
+   \nabla u^m \mathcal{N} \cdot \Delta u^m \right] \, dx \\
=&\frac{1}{2} \int_{\partial \Omega_F} \nabla u^m : \nabla u^m dS
-\frac{1}{2}\int_{\Omega_F} \left(\nabla \cdot \mathcal{N}\right) \nabla u^m : \nabla u^m \, dx \\
&+\int_{\Omega_F} \nabla u^m \nabla \mathcal{N} :\nabla u^m \, dx
+\frac{1}{\eta} \int_{\Omega_F} \left(\nabla u^m  \mathcal{N}\right) \cdot \nabla p^m \, dx.
\end{split}
\end{equation*}
Integrating by parts once more and using the fact that $\nabla \cdot u^m=0$, we get that
\begin{align}
\int_{\partial \Omega_F} \left| \frac{\partial u^m}{\partial n} \right|^2 dS
=& \frac{1}{2} \int_{\partial \Omega_F} \nabla u^m : \nabla u^m dS \nonumber \\
&+ \int_{\Omega_F} \left[ \nabla u^m \nabla \mathcal{N}: \nabla u^m
-\frac{1}{\eta} \nabla u^m : \left(\nabla \mathcal{N}\right)^T p^m \right] \, dx \nonumber \\
&+\frac{1}{\eta} \int_{\partial \Omega_F} \nabla u^m : n n p^m dS
- \frac{1}{2} \int_{\Omega_F} \left(\nabla \cdot \mathcal{N}\right) \nabla u^m :\nabla u^m \, dx \nonumber \\
=:& I_1 + I_2 + I_3 + I_4. \label{dudn_estimate_1}
\end{align}
Since $\nabla \cdot \mathcal{N} \geq 0$ a.e., we can neglect $I_4$.
We can decompose
$\nabla u^m = \nabla u^m n n + \sum_{i=1}^2 \nabla u^m \tau^i \tau^i$
to estimate
\begin{equation}
\label{dudn_estimate_2}
I_1 = \frac{1}{2} \left(\int_{\partial \Omega_F} \left| \frac{\partial u^m}{\partial n} \right|^2 dS +
\sum_{i=1}^2 \int_{\partial \Omega_F} \left| \frac{\partial u^m}{\partial \tau^i} \right|^2 dS. \right)
\end{equation}
We can use Cauchy's inequality to estimate, for all $\delta_1 > 0$,
\begin{equation*}
I_3 \leq \frac{1}{4 \delta_1} \int_{\partial \Omega_F} \left| \frac{\partial u^m}{\partial n} \cdot n \right|^2 dS +
\frac{\delta_1}{\eta^2} \| p^m \|^2_{\partial \Omega_F}.
\end{equation*}
Applying Lemma \ref{incompressibility_at_bdy}, we get that
\begin{equation}
\label{dudn_estimate_3}
I_3 \leq \frac{1}{2 \delta_1} \sum_i \int_{\partial \Omega_F} \left| \frac{\partial u^m}{\partial \tau^i} \right|^2 dS +
\frac{\delta_1}{\eta^2} \| p^m \|^2_{\partial \Omega_F}.
\end{equation}
We can also use equation \eqref{nabla_N_L_infty} and Cauchy's inequality to get, for all $\delta_2 > 0$,
\begin{equation}
\label{dudn_estimate_4}
\begin{split}
I_2 \leq & \frac{1}{d/2-a} \left[ \sum_{j,k} \int_{\Omega_F} \left| \frac{\partial u^m}{\partial x_j} \cdot \frac{\partial u^m}{\partial x_k} \right| \, dx
+ \frac{1}{\eta} \sum_{i,j} \int_{\Omega_F} \left| \frac{\partial u^m_i}{\partial x_j} p^m \right| \, dx \right] \\
\leq & \frac{1}{d/2-a} \left[ \left(1+\frac{1}{4\delta_2}\right) \| \nabla u^m \|^2_{\Omega_F} + \delta_2 \frac{1}{\eta^2} \|p^m \|^2_{\Omega_F} \right].
\end{split}
\end{equation}
Integrating by parts and using the Stokes equation,
\begin{equation*}
\begin{split}
\| \nabla u^m \|^2_{\Omega_F} 
=& -\int_{\Omega_F} u^m \cdot \Delta u^m \, dx + \int_{\partial \Omega_F} u^m \cdot \frac{\partial u^m}{\partial n} dS \\
=& -\frac{1}{\eta}\int_{\Omega_F} u^m \cdot \nabla p^m \, dx + \int_{\partial \Omega_F} u^m \nabla u^m \cdot n dS.
\end{split}
\end{equation*}
Integrating by parts once more and using the fact that $\nabla \cdot u^m=0$, we get that
\begin{equation*}
\begin{split}
\| \nabla u^m \|^2_{\Omega_F}
=& \frac{1}{\eta} \int_{\partial \Omega_F} u^m \cdot n p^m dS + \int_{\partial \Omega_F} u^m \cdot \frac{\partial u^m}{\partial n} dS.
\end{split}
\end{equation*}
Using Cauchy's inequality twice more, we get that
\begin{align}
\| \nabla u^m \|^2_{\Omega_F} \leq & \frac{1}{4 \delta_3} \| u^m \|^2_{\partial \Omega_F} + \frac{\delta_3}{\eta^2} \| p^m \|^2_{\partial \Omega_F}
+ \frac{1}{4 \delta_4} \| u^m \|^2_{\partial \Omega_F} + \delta_4 \left\| \frac{\partial u^m}{\partial n} \right\|^2_{\partial \Omega_F}. \label{dudn_estimate_5}
\end{align}
Combining estimates \eqref{dudn_estimate_1}, \eqref{dudn_estimate_2}, \eqref{dudn_estimate_3}, \eqref{dudn_estimate_4}, and \eqref{dudn_estimate_5}, and taking the limit as $m\rightarrow \infty$ (all terms converge to the desired quantities due to Stokes regularity--see, for example, Theorem IV.6.1 in \cite{galdi}--and the trace theorem), we get equation \eqref{dudn_estimate}.

\end{proof}

\subsection{Estimate of $\| p \|_{\Omega_F}$}
\label{p_interior_estimate_section}

We now turn to proving the pressure estimate \eqref{p_inequality_preliminary}.
We will normalize $p$, which is defined only up to an additive constant, in order to optimize constants in that inequality.

Since best constants in trace and Sobolev inequalities are difficult to obtain  in general, we proceed as follows. We first construct an explicit extension $\tilde{p} \in H^{1/2}(\Omega_F)$ of the trace of $p$ in   $L^2(\partial\Omega_F)$, for which we have good control on the trace constants.
Then, since $p$ is harmonic from Stokes equation, $p = \tilde{p}-p_0$, where $p_0$ solves
\begin{equation}
\label{pde_p}
\left\{
\begin{array}{lr}
\Delta p_0 = \Delta \tilde{p} & x \in \Omega_F \\
p_0 = 0 & x\in\partial \Omega_F. \\
\end{array}
\right.
\end{equation}
Hence, \  $\|p \|_{\Omega_F} \leq \| \tilde{p} \|_{H^{1/2}(\Omega_F)} + \| p_0 \|_{\Omega_F}$.
Then, \ $\|p \|_{\Omega_F} \leq (1+C_1) C_2 \|p \|_{\partial\Omega_F}$, where  $C_1$, given in Theorem \ref{thm_pde_p}, is the constant in \
$\| p_0 \|_{\Omega_F} \leq C_1 \| \tilde{p} \|_{H^{1/2}(\Omega_F)}.
$ 
and
$C_2$, given in Theorem \ref{thm_extension}, is the constant in \
$\| \tilde{p} \|_{H^{1/2}(\Omega_F)} \leq C_2 \| p \|_{\partial\Omega_F}.
$ 

\begin{lemma}
\label{pde_p_lemma_1}
Let $\tilde{p} \in H^{1/2}\cap L^\infty(\Omega_F)$ and let $p_0\in H^{1/2}(\Omega_F)$ be the solution of equation \eqref{pde_p}.
Then
\begin{equation}
\label{estimate_pde_p_L^2}
\| p_0 \|_{\Omega_F} \leq \left\| \tilde{p} \right\|_{\Omega_F} + \sqrt{| \Omega_F|} \left\| \tilde{p} \right\|_{L^\infty(\partial \Omega_F)}.
\end{equation}
\end{lemma}

\begin{proof}
Let $\tilde{p}_n \in C^\infty(\Omega_F)$ be such that $\tilde{p}_n \rightarrow \tilde{p}$ in $H^{1/2}(\Omega_F)$ and let $p^0_n$ be the (smooth) solution of \eqref{pde_p} with data $\tilde{p}_n$.
Then, by standard regularity results (see, for example, \cite{lions}), $p^0_n \rightarrow p_0$ in $H^{1/2}(\Omega_F)$.  Let $\mathcal{E}:= \{ \phi \in C^\infty_0(\Omega_F), \|\Delta \phi \|_{\Omega_F}=1\}$.\
Using  the dual characterization of the $L^2$-norm and the density of $C^\infty_0$ in $L^2$ gives
\begin{align}
\| p^0_n \|_{\Omega_F}
&= \sup_{\phi \in \mathcal{E}} \int_{\Omega_F} \Delta \tilde{p}_n  \phi \, dx
= \sup_{\phi \in \mathcal{E}} \left\{ -\int_{\Omega_F} \nabla \tilde{p}_n \cdot \nabla \phi \, dx  \right\}\nonumber \\
&= \sup_{\phi \in \mathcal{E}}  \left\{\int_{\Omega_F} \tilde{p}_n \Delta \phi \, dx - \int_{\partial \Omega_F} \tilde{p}_n \frac{\partial \phi}{\partial n} \, dS \right\}\nonumber \\
& \leq \left\| \tilde{p}_n \right\|_{\Omega_F} + \left\| \tilde{p}_n \right\|_{L^\infty(\partial \Omega_F)} \sup_{\phi \in \mathcal{E}} \int_{\partial \Omega_F} \left| \frac{\partial \phi}{\partial n} \right| dS.
\end{align}
We next note that for any $\phi$ in the admissible class $\mathcal{E}$, one can find a sequence $\phi_k \in \mathcal{E}$ such that $\frac{\partial \phi_k}{\partial n} \geq 0$ and $\frac{\partial \phi_k}{\partial n} \rightarrow \left|\frac{\partial \phi}{\partial n}\right|$ on $\partial \Omega_F$ in $L^1(\partial \Omega_F)$.  Therefore, we can assume $\frac{\partial \phi}{\partial n} \geq 0$ without affecting the supremum. Therefore,
\begin{equation*}
\begin{split}
\| p^0_n \|_{\Omega_F} & \leq \left\| \tilde{p}_n \right\|_{\Omega_F} + \left\| \tilde{p}_n \right\|_{L^\infty(\partial \Omega_F)} \sup_{\phi \in \mathcal{E} \text{, } \frac{\partial \phi}{\partial n} \geq 0 } \int_{\partial \Omega_F} \frac{\partial \phi}{\partial n} dS \\
& =\left\| \tilde{p}_n \right\|_{\Omega_F} + \left\| \tilde{p}_n \right\|_{L^\infty(\partial \Omega_F)} \sup_{\phi \in \mathcal{E} \text{, } \frac{\partial \phi}{\partial n} \geq 0 } \int_{\Omega_F} \Delta \phi \, dx \\
& \leq \left\| \tilde{p}_n \right\|_{\Omega_F} + \sqrt{| \Omega_F|} \left\| \tilde{p}_n \right\|_{L^\infty(\partial \Omega_F)}.
\end{split}
\end{equation*}
Since $p^0_n\rightarrow p_0$ and $\tilde{p}_n  \rightarrow \tilde{p}$ respectively in
$L^2(\Omega_F)$ and $L^\infty(\partial\Omega_F)$, estimate \eqref{estimate_pde_p_L^2}
follows.
\end{proof}

\begin{theorem}
\label{thm_pde_p}
Let $\tilde{p} \in H^{1/2}\cap L^\infty(\Omega_F)$ and $p_0\in H^{1/2}(\Omega_F)$ be the solution of equation \eqref{pde_p}.
Then there exists a constant $C>0$ such that  \
$\| p_0 \|_{\Omega_F} \leq \sqrt{2} \| \tilde{p} + C \|_{\Omega_F}.
$ 
\end{theorem}

\begin{proof}
We can add an arbitrary constant to $\tilde{p}$ without changing \eqref{pde_p}, and hence its solution $
p_0$. We hence can  substitute  $\tilde{p}^+ := \tilde{p} + 3 \| \tilde{p} \|_{L^\infty(\Omega_F)}$ for $\tilde{p}$ in \eqref{estimate_pde_p_L^2} to obtain:
\begin{align}
\| \tilde{p}^+ \|^2_{\Omega_F} =&
\int_{\Omega_F} \left( \tilde{p}^2+ 6 \tilde{p} \| \tilde{p} \|_{L^\infty(\Omega_F)} + 9 \| \tilde{p} \|_{L^\infty(\Omega_F)}^2 \right) \, dx 
\geq \| \tilde{p} \|^2_{\Omega_F} +3 \left|\Omega_F\right| \| \tilde{p} \|^2_{L^\infty(\Omega_F)} \nonumber \\
\geq & \frac{1}{2} \left( \left\| \tilde{p} \right\|_{\Omega_F} + \sqrt{| \Omega_F|} \left\| \tilde{p} \right\|_{L^\infty(\partial \Omega_F)} \right)^2
\geq \frac{1}{2} \left\| p_0 \right\|^2_{\Omega_F}. \nonumber
\end{align}

\end{proof}

In order to prove Theorem \ref{thm_extension}, we will need the following Lemmas.
We include  proofs for the reader's sake.

\begin{lemma}
\label{extension_lemma_1}
Let $p\in H^{1/2}(\partial B(0,R))$ and let $\tilde{p}^I\in H^{1}(B(0,R))$ solve
\begin{equation}
\label{p_I_pde}
\left\{
\begin{array}{lr}
\Delta \tilde{p}^I = 0 & x\in B(0,R) \\
\tilde{p}^I = p & x \in \partial B(0,R).
\end{array}
\right.
\end{equation}
Then
$\| \tilde{p}^I \|_{B(0,R)} \leq \sqrt{\frac{R}{3}} \| p \|_{\partial B(0,R)}
$ 
and
$\| \nabla \tilde{p}^I \|_{B(0,R)} \leq | p |_{H^{1/2} (\partial B(0,R))}.
$ 
\end{lemma}

\begin{proof}
We first recall that for any $s\geq 0$, linear combinations of spherical harmonics are dense in $H^s(S^2)$, where $S^2 = \partial B(0,1)$ is the unit sphere,  and that an equivalent norm is given by
\begin{equation} \label{harmonics1}
   \|f\|_{H^s(S^2)}  = \left ( \sum_{n=0}^\infty \sum_{k=-n}^n \lambda_n^{2s} |G_{nk}|^2\right )^{1/2},
\end{equation}
where $\lambda_n = n\,(n+1)$, and $A_{nk}$ is the $L^2$ projection of $f$ onto $Y^k_n$, a spherical harmonic of degree $n$ (we refer, for instance, to \cite{lions}).
By rescaling  a similar expansion is valid for $p\in L^2(\partial B(0,R)$ for any fixed $R>0$.
Moreover, if we set  $p^R(\omega) := p(R\, \omega)$, $\omega \in S^2$,
\begin{equation} \label{harmonics2}
    \|p\|_{H^s(\partial B(0,R))}  = R^{1-s} \, \|p^R\|_{H^s(S^2)}.
\end{equation}

By density, we can assume therefore that  $p$ is a finite linear combination of spherical harmonics:  \ $ p =\sum_{n=0}^N \sum_{k=-n}^n   G_{nk}(R)\,  Y_{n,k}$.  The solution $p^I$ of the Dirichlet problem \eqref{p_I_pde} can then be written (see for example \cite[Theorem 2.10, p. 145]{SteinWeiss}), using the homogeneity of the Laplace operator, as  \
$\tilde{p}^I = \sum_{n=0}^N  \sum_{k=-n}^n R^{-n}\, G_{nk}(R) \,  r^n Y_n^k$, which is in fact an harmonic polynomial in the whole of $\mathbb{R}^3$.

Next, we observe that, if $(r, \theta,\phi)$ are spherical coordinates, then
\[
  \partial_n \tilde{p}^I = \partial_r \tilde{p}^I = \sum_{n=0}^N  \sum_{k=-n}^n R^{-n}\, G_{nk}(R) \,  n r^{n-1} Y_n^k
\]
where $n$ is the unit outer normal to  $\partial B(0,R)$.
Integrating by parts and by the orthonormality of the spherical harmonics, we hence have:
\begin{equation}
\label{tilde_p_I_H^1_bdy}
\begin{split}
\big\| \nabla \tilde{p}^I & \big\|^2_{B(0,R)} = \int_{\partial B(0,R)} \frac{\partial \tilde{p}^I}{\partial n} \, \overline{\tilde{p}^I} \, dS
 = \int_{S^2} \left(
\sum_{l = 0}^N \sum_{m=-l}^l  G_{lm}(R) Y_l^m \right) \cdot \\
&\cdot  \left( \sum_{n = 0}^N \sum_{q=-n}^n \overline{G_{nq}(R)} R^{-n} n R^{n-1} \overline{Y_n^q}  \right) \,R^2\, dS
= 
\sum_{n=0}^N \sum_{q=-n}^n n R^{-1}\, |G_{nq}(R)|^2.
\end{split}
\end{equation}
Additionally, from \eqref{harmonics1}--\eqref{harmonics2},
\ $\left\| \tilde{p}^I \right\|^2_{B(0,R)}
= \sum_{l=0}^N \sum_{m=-l}^l |G_{lm}(R)|^2  \frac{R^3}{3+2l}$, \
$\left\| p \right\|^2_{\partial B(0,R)} = \sum_{l=0}^N \sum_{m=-l}^l |G_{lm}(R)|^2 \, R^2$, \
and \quad
$\left| p \right|^2_{H^{1/2}(\partial B(0,R))} = \newline
\sum_{l=0}^\infty \sum_{m=-l}^l G^2_{lm} \sqrt{l(l+1)} R$.
\ Since $l/\sqrt{ l (l+1)} \leq 1$ for $l \geq 0$, we get the desired estimate for $p$ a finite sum of spherical harmonics. The estimate for any given $p \in H^{1/2}(\partial B(0,R))$ then follows by density, using the continuity in Sobolev spaces of solutions to \eqref{p_I_pde}.
\end{proof}

\begin{lemma}
\label{extension_lemma_2}
Let $p\in H^{-1/2}(\partial B(0,R))$ and let $\tilde{p}^I$ solve equation \eqref{p_I_pde}.
Then $\tilde{p}^I \in  L^2(B(0,R))$ and $\| \tilde{p}^I \|_{B(0,R)} \leq \sqrt{\frac{1}{2}} \| p \|_{H^{-1/2}(\partial B(0,R))}$.
\end{lemma}

The proof of the Lemma is very similar to that of Lemma \ref{extension_lemma_1}.

\begin{lemma}
\label{extension_lemma_3}
Let $p\in L^2(\partial B(0,R))$ and let $\tilde{p}^I\in H^{1/2} (B(0,R))$ solve equation \eqref{p_I_pde}.
Then, $\| \tilde{p}^I \|_{H^{1/2}(B(0,R))} \leq \left( \frac{1}{2} \right)^{\frac{1}{4}} \left(1+\frac{R}{3} \right)^{\frac{1}{4}} \| p \|_{\partial B(0,R)}$.
\end{lemma}

\begin{proof}
The result follows by interpolating the estimates in  Lemmas \ref{extension_lemma_1} and \ref{extension_lemma_2} (see e.g. Theorem A.1.1 in \cite{DaPrato} ).
\end{proof}

\begin{lemma}
\label{extension_lemma_6}
Let $p\in L^2(\partial B(0,a))$.  Then there exists an extension $\tilde{p}^E\in H^{1/2}(B^c(0,a)\cap B(0,d))$, where $d>a$, such that \ $\tilde{p}^E \big|_{\partial B(0,a)} = p$
satisfying \quad
$
\| \tilde{p}^E \|_{H^{1/2}(B^c(0,a)\cap B(0,d))} \leq 2^{\frac{1}{8}}\left(d+\frac{1}{a}+\sqrt{2}\right)^{\frac{1}{4}} \| p \|_{\partial B(0,a)}.
$

\end{lemma}

\begin{proof}
The proof is similar to the combined proofs of Lemmas \ref{extension_lemma_1}--\ref{extension_lemma_3}.
\end{proof}

To obtain an  appropriate extension of $p$  in  $\Omega_F$ from Lemmas \ref{extension_lemma_1} -- \ref{extension_lemma_6}, we introduce the cutoff functions
\begin{align}
\zeta^E (r) &:= \left\{
\begin{array}{lr}
e \text{exp} \left[ -\frac{1}{1-\left(\frac{r-a}{d-a}\right)^2} \right], &  a \leq r \leq d, \label{zeta^E_def} \\
0 & r > d,
\end{array}
\right. \\
\zeta^I (r) &:= \left\{
\begin{array}{lr}
e \text{exp} \left[ -\frac{1}{1-\left(\frac{R-r}{R-d}\right)^2} \right], & d \leq r \leq R, \label{zeta^I_def} \\
0 & r < d.
\end{array}
\right.
\end{align}
One can easily calculate that
\begin{equation} \label{zeta_E_L^infty}
\begin{aligned}
&\| \zeta^E \|_{L^\infty( B^c(0,a))}  \leq 1 , \quad
&\| \nabla \zeta^E \|_{L^\infty( B^c(0,a))}  \leq \frac{3}{d}, \\
&\| \zeta^I \|_{L^\infty( B(0,R))}  \leq 1,   \quad
&\| \nabla \zeta^I \|_{L^\infty( B(0,R))}  \leq \frac{3}{d}.
\end{aligned}
\end{equation}

The next two results follow easily from \eqref{zeta_E_L^infty}.

\begin{lemma}
\label{multiplication_by_zeta_E_lemma}
Let $V(a,R):=B^c(0,a)\cap B(0,R)$ with $a<d<R$ and $\tilde{p}^E \in H^{1/2}(V(a,R))$.  Then $\tilde{p}^E \zeta^E \in H^{1/2}(B^c(0,a))$ with
\begin{equation*}
\| \tilde{p}^E \zeta^E \|_{H^{1/2}(B^c(0,a))} \leq \left(1+\frac{3}{d}\right)^{\frac{1}{2}} \|\tilde{p}^E \|_{H^{1/2}(V(a,R))}.
\end{equation*}
\end{lemma}


\begin{lemma}
\label{multiplication_by_zeta_I_lemma}
Let $\tilde{p}^I \in H^{1/2}(B(0,R))$ with $d<R$.  Then
$\tilde{p}^I \zeta^I \in H^{1/2}(B(0,R))$
with
$\| \tilde{p}^I \zeta^I \|_{H^{1/2}(B(0,R))} \leq \left(1+\frac{3}{d}\right)^{\frac{1}{2}} \|\tilde{p}^I \|_{H^{1/2}(B(0,R))}.$
\end{lemma}


We are now ready to state the extension theorem for $p$.

\begin{theorem}
\label{thm_extension}
Let $0<a<d<R\in \R$ and $\Omega_F:= B(0,R) \setminus \Omega_B$, where $\Omega_B = \bigcup_{l\in \Lambda} B(x^l, a)$, $x^l \in d\mathbb{Z}^3$,
and $\Lambda=\left\{ l \in \mathbb{Z}^3 \big| \text{ } |x^l| + \frac{d}{2}-a < R \right\}$.  Let $p \in L^2(\partial \Omega_F)$.  Then there exists an extension $\tilde{p}\in H^{1/2}(\Omega_F)$ such that $\tilde{p}\big|_{\partial \Omega_F} = p$, which satisfies
$\| \tilde{p} \|_{H^{1/2}(\Omega_F)} \leq
\beta \| p \|_{\partial \Omega_F},$
where $\beta$ is given in equation \eqref{beta}.

\end{theorem}

\begin{proof}
Define
$\displaystyle
\tilde{p}(x) = \sum_{l=1}^N \tilde{p}^E (x-x^l) \zeta^E (x-x^l) + \tilde{p}^I (x) \zeta^I (x),
$
where $\tilde{p}^E$ and $\tilde{p}^I$ are given  in Lemmas \ref{extension_lemma_6} and \ref{extension_lemma_3}, respectively, and $\zeta^E$ and $\zeta^I$ are the cutoffs defined in equations \eqref{zeta^E_def} and \eqref{zeta^I_def}, respectively.  Then $\tilde{p}\big|_{\partial \Omega_F} = p$ and, from Lemmas \ref{extension_lemma_3} 
-- \ref{multiplication_by_zeta_I_lemma} it follows that
\begin{align}
\| \tilde{p} \|_{H^{1/2}(\Omega_F)}
\leq & \sum_{l=1}^N \left(1+\frac{6}{d}\right)^{\frac{1}{2}} 2^{1/8} \left(\frac{d}{2}+\frac{1}{a}+\sqrt{2} \right)^{\frac{1}{4}} \| p \|_{\partial B(x^l,a)} \nonumber \\
&+ \left(1+\frac{3}{R-d/2}\right)^{\frac{1}{2}} \left(\frac{1}{2}\right)^{\frac{1}{4}} \left( 1+\frac{R}{3} \right)^{\frac{1}{4}} \| p \|_{\partial B(0,R)}
\leq
\beta \|p \|_{\partial \Omega_F}. \nonumber
\end{align}
\end{proof}

Combining Theorems \ref{thm_pde_p} and \ref{thm_extension}, we finally obtain \eqref{renormalized_p_ineq}.

The results in this section can be extended to other shapes through the use of eigenfunctions of the Laplace operator on their surfaces.  For many simple shapes, including ellipsoids, these are explicitly known \cite{hobson}.

\subsection{Estimate of $\| p \|_{\partial \Omega_F}$}
\label{p_bdy_estimate_section}

In this section, we establish the last needed estimate \eqref{p_bdy_estimate_preliminary}.
We will do so by employing the boundary integral equation for $p$ derived from the Stokes system, where we recall \ $\mathcal{P}(\nu)  = \nu/(4\pi\, |\nu|^3)$ (see, for example, \cite{pozrikidis}):
\begin{equation}
\label{bie_p}
\frac{1}{2} p(\xi)=-\frac{1}{2} \text{p.v.} \, \int_{\partial \Omega_F} \mathcal{P}(\xi-\nu) \cdot t(\nu) \, dS_\nu
-\eta \nabla_{\xi} \cdot  \text{p.v.}\, \int_{\partial \Omega_F} n \mathcal{P}(\xi-\nu) \cdot u(\nu) \, dS_\nu.
\end{equation}
To simplify notation, we will henceforth drop the p.v.~from principal value integrals.  Any integral involving $\mathcal{P}$ in this section should be interpreted in a principal value sense.

To explicitly estimate the norm of the kernel $\mathcal{P}$ in each integral, we employ the Fourier transform. To this end,
we embed $\R^3$ into $\C^3$ and thence into $\mathcal{A}_3$, the $8$-dimensional Clifford algebra $Cl_{0,3}$
with basis $1,e_1,e_2,e_3,e_1 e_2, e_2 e_3, e_1 e_3, e_1 e_2 e_3$.
For simplicity, we will write $e_{ij}:= e_i e_j$ and $e_{123}:= e_1 e_2 e_3$.
Under this embedding, scalars are identified with multiples of $1$ and vectors
$x=(x_1,x_2,x_3)\in \C^3$ with $x_1 e_1 + x_2 e_2 + x_3 e_3$.
For an arbitrary $x=x_0 + x_1 e_1 + \cdots x_{12} e_1 e_2 + \cdots \in \mathcal{A}_3$,
 we let
\begin{align}
\overline{x}& := x_0 - x_1 e_1 - x_2 e_2 - x_3 e_3 - x_{12} e_{12} - x_{23} e_{23} - x_{13} e_{13} + x_{123} e_{123}, \nonumber \\
\widetilde{x}& := x_0 + x_1 e_1 + x_2 e_2 + x_3 e_3 - x_{12} e_{12} - x_{23} e_{23} - x_{13} e_{13} - x_{123} e_{123}, \nonumber \\
x^\dagger &:= x_0^* - x_1^* e_1 - x_2^* e_2 - x_3^* e_3 - x_{12}^* e_{12} - x_{23}^* e_{23} - x_{13}^* e_{13} + x_{123}^* e_{123}, \nonumber
\end{align}
where $x_i^*$ denotes the complex conjugate of $x_i$.

Let $\Gamma_1,\Gamma_2 \subset \R^3$ be spheres or planes.
By $L^2(\Gamma_k, \mathcal{A}_3)$
we denote the space of functions  $v:\Gamma_k \rightarrow \mathcal{A}_3$ such that
$\| \int_{\Gamma_k} v(x) v^\dagger(x) dS_x \| < \infty$,
which is a Hilbert space with the inner product
\begin{equation*}
\left< f, g \right>_{\Gamma_k} := \int_{\Gamma_k} f^\dagger(x) g(x) \, dS_x.
\end{equation*}
Since this definition is valid for all $L^2(\Gamma_k, \mathcal{A}_3)$ functions, in which $C^\infty(\Gamma_k, \mathcal{A}_3)$ are dense, we can extend this definition to the case where either $f$ or $g$ is a distribution.
Under the linear fractional transformation $\Phi(x)=(ax+b)(cx+d)^{-1}$ mapping $\Gamma_2$ to $\Gamma_1$, we can write
\begin{equation}
\label{cov_formula}
\left< f(y), g(y) \right>_{\Gamma_1} = \left< J(\Phi, x) f(\Phi(x)), J(\Phi, x) g(\Phi(x)) \right>_{\Gamma_2},
\end{equation}
where $J(\Phi,x):= (\widetilde{cx} + \widetilde{d}) \|cx+d\|^{-3}$.

We now assume $\Gamma=S^2$ and set $\phi(x) = \left( e_1 x + 1\right) \left(x + e_1\right)^{-1}$.  Then, $\psi = \phi^{-1}$ is the standard stereographic projection from $S^2$ onto $\mathbb{R}^2$, identified with the  span over $\mathbb{R}$ of $e_1,e_2$.
This projection allows one to define the Fourier transform for
functions $f \in L^2(S^2, \mathcal{A}_3)$ as:
\begin{equation*}
\mathcal{F}_{S^2} (f) (z) := \frac{1}{2\pi} J(\psi, z)  \int_{S^2} e^{i \left<\psi(x),\psi(z)\right>} \overline{J}(\psi, x) f(x) dS_x,
\end{equation*}
where $J(\psi,x) = 4  \left( \widetilde{x} - e_1 \right) \| x + e_1 \|^{-3}$ and $\left< x, y\right>:= \sum_{i=1}^3 x_i y_i$. $\mathcal{F}_{S^2}$ is an isometry from $L^2(S^2, \mathcal{A}_3)$ into itself (see e.g. \cite{ryan}.)

\begin{theorem}
\label{pi_norm_sphere}
Let $t \in (L^2(\partial B(0,r)))^3$ and
\begin{equation}
\label{Pi_def}
\Pi_{\partial B(0,r)} t := \int_{\partial B(0,r)} \mathcal{P}(\xi-\nu) \cdot t(\nu) dS_\nu.
\end{equation}
Then, $\Pi_{\partial B(0,r)}$ is a bounded operator on  $L^2(\partial B(0,r))$
with norm $\| \Pi_{\partial B(0,r)} \|\leq \frac{1}{2}$.
\end{theorem}

\begin{proof}
We remind the reader that the integral in \eqref{Pi_def} must be understood in a principal value sense.
This integral is invariant under dilations,
so it is sufficient to consider only the case $r=1$. Moreover, it is enough to calculate the $L^2$-norm of $\mathcal{F}_{S^2} \left(\Pi_{S^2} t\right)$, as the Fourier transform is an isometry.  Let $t_m\in C^\infty(\Gamma_k, \mathcal{A}_3)$ be such that $t_m \rightarrow t$ in $L^2(\Gamma_k, \mathcal{A}_3)$.
Then, noting that (see, for example, \cite{peetre})
$\mathcal{P}\left(\psi(\xi)-\psi(\nu)\right)= J(\psi, \xi)^{-1} \mathcal{P}(\xi-\nu) \widetilde{J}(\psi, \nu)^{-1}$,
we have
\begin{equation*}
\begin{split}
\Pi_{S^2}  t_m(\xi) =& \int_{S^2} \mathcal{P}(\xi-\nu) t_m(\nu) d S_\nu
= \int_{S^2} J(\psi,\xi) \mathcal{P}\left(\psi(\xi)-\psi(\nu)\right) \widetilde{J}(\psi,\nu) t_m(\nu) d S_\nu. \\
\intertext{The above equality is valid in principal value sense (i.e., pointwise outside a neighborhood of radius $\epsilon$ about $\xi=\nu= -e_1$, since $J$ is singular there, and then passing to the limit $\epsilon\rightarrow 0$).  Since $\widetilde{J}=J$, and $J^\dagger=\overline{J}$ (because the components of $J$ are real-valued), this is}
=& J(\psi, \xi) \left< J(\psi, \nu) J(\psi, \nu)^{-1} \left(\mathcal{P}(\psi(\xi)-\psi(\nu))\right)^\dagger,
J(\psi, \nu) t_m(\nu) \right>_{S^2}.
\end{split}
\end{equation*}
By equation \eqref{cov_formula}, we obtain
\begin{equation*}
\begin{split}
\Pi_{S^2} t_m (\xi)
=& J(\psi, \xi)\left< J(\psi, \phi(\nu))^{-1} \left(\mathcal{P}(\Psi(\xi)-\nu)\right)^\dagger,
 t(\phi(\nu)) \right>_{\R^2}. \\
\end{split}
\end{equation*}
 We observe that, since $J J^{-1}=1$, $J^{-1}(\psi,x)$ is bounded over the sphere and vanishes for $x=-e^1$, so that $J(\psi, \phi(\nu))^{-1} \left(\mathcal{P}(\Psi(\xi)-\nu)\right)^\dagger$ is well defined as a principal value distribution.  Therefore, this and the following expressions are well-defined.
Next, we compute the Fourier transform of this expression:
\begin{equation*}
\begin{split}
\mathcal{F}_{S^2} \left(\Pi_{S^2} t_m\right) (z)
=& \frac{1}{2\pi} J(\psi,z) \bigg< J(\psi,\xi) e^{-i \left< \psi(\xi),\psi(z) \right>},  \\
&  J(\psi,\xi) \int_{\R^2} \mathcal{P}(\psi(\xi)-\nu) \overline{J}(\psi, \phi(\nu))^{-1} t_m(\phi(\nu))  d \nu \bigg>_{S^2}. \\
\end{split}
\end{equation*}
Denoting convolution by $\ast$, it follows  from  equation \eqref{cov_formula},
\begin{equation} \label{FT_Pi_t}
\begin{split}
\mathcal{F}_{S^2} \left(\Pi_{S^2} t_m\right) (z)
=& \frac{1}{2\pi} J(\psi,z) \bigg<e^{-i \left< \xi,\psi(z) \right>},
\int_{\R^2} \mathcal{P}(\xi-\nu) \overline{J}(\psi, \phi(\nu))^{-1} t_m(\phi(\nu))  d \nu \bigg>_{\R^2} \\
=& J(\psi,z) \mathcal{F}_{\R^2} \left( \mathcal{P}(\xi) * \overline{J}(\psi,\phi(\xi))^{-1} t_m(\phi(\xi)) \right) (\psi(z)), \\
=& J(\psi,z) \mathcal{F}_{\R^2} \left( \mathcal{P}(\xi)\right)(\psi(z)) \mathcal{F}_{\R^2} \left(\overline{J}(\psi,\phi(\xi))^{-1} t_m(\phi(\xi)) \right) (\psi(z)).
\end{split}
\end{equation}
We now observe that, since $J(\phi,\psi(\xi))^{-1} = \overline{J}(\psi,\xi)$ for $\xi\in \R^2$,
\begin{equation} \label{FT_JinvT}
\begin{split}
\mathcal{F}_{\R^2} \big(\overline{J}& (\psi,\phi(\xi))^{-1} t(\phi(\xi)) \big) (\psi(z))
= \frac{1}{2\pi} \int_{\R^2} e^{i \left< \xi, \psi(z) \right>} J(\phi,\xi) t(\phi(\xi)) d\xi \\
=& -\frac{1}{2\pi} \left<J(\phi,\xi) e^{-i \left<\xi,\psi(z) \right>}, J(\phi,\xi) J(\phi,\xi)^{-1} t(\phi(\xi)) \right>_{\R^2} \\
=& -\frac{1}{2\pi} \int_{S^2} e^{i \left<\psi(\xi),\psi(z)\right>} \overline{J}(\psi,\xi) t(\xi) dS_\xi \\
=& -J(\psi,z)^{-1} \mathcal{F}_{S^2}\left(t\right)(z),
\end{split}
\end{equation}
again by \eqref{cov_formula}.
In addition,  $\mathcal{P}$ if a Fourier multiplier with bounded symbol $\mathcal{F}_{\R^2} \left(\mathcal{P}(\xi) \right)$, since (see e.g. \cite{mcintosh}),
\begin{equation}
\label{FTr_P}
\mathcal{F}_{\R^2} \left(\mathcal{P}(\xi) \right) (\nu)= -\frac{i}{2} \frac{\nu}{\|\nu \|},
\end{equation}
and, therefore, we can pass to the limit $m \rightarrow \infty$ in equation \eqref{FT_Pi_t}.
Finally, we combine this equation with \eqref{FT_JinvT} and  \eqref{FTr_P}, and observe that  for $z\in\R^3$, $J(\psi,z)^{-1} = -J(\psi,z) \| J(\psi,z) \|^{-2}$, so that
\begin{equation*}
 \begin{aligned}
\mathcal{F}_{S^2} \left(\Pi_{S^2} t\right) (z) & = \frac{i}{2} J(\psi,z) \frac{\psi(z)}{\| \psi(z) \|} J(\psi,z)^{-1} \mathcal{F}_{S^2}\left(t\right)(z)\\
&= -\frac{i}{2} \frac{J(\psi,z)}{\| J(\psi,z) \|} \frac{\psi(z)}{\| \psi(z) \|} \frac{J(\psi,z)}{\| J(\psi,z) \|} \mathcal{F}_{S^2}\left(t\right)(z).
\end{aligned}
\end{equation*}
\end{proof}

\begin{theorem}
\label{Pi_estimate_thm}
Let $\Gamma=\bigcup_l \partial B(x^l,a^l)$, where $\partial B(x^m,a^m) \cap \partial B(x^n,a^n)= \emptyset$ $\forall m \neq n$.
Let $t \in L^2(\partial \Gamma,\mathcal{A}_3)$.
Then, as an operator from $L^2(\Gamma,\mathcal{A}_3)\rightarrow L^2(\Gamma,\mathcal{A}_3)$, $\| \Pi_{\Gamma} \|\leq \frac{1}{2}$.
\end{theorem}

\begin{proof}
We introduce the mapping
\ $
\mathcal{F}^B_\Gamma ( f ) (z):= \sum_l \mathcal{F}_{\partial B(x^l, a^l)} (f) (z) \chi_{\partial B(x^l, a^l)},
$ \ where $\chi_{F}$ is the characteristic function of the set $F$.
We observe that \ $L^2(\Gamma) \cong \oplus_{l} L^2(\partial B(x^l, a^l)$ and by construction $\mathcal{F}^B_\Gamma ( f ) \cong \oplus_{l} \mathcal{F}_{\partial B(x^l, a^l)}$.
Consequently,  $\mathcal{F}^B_\Gamma$ is an isometry on $L^2(\Gamma,\mathcal{A}_3)$ as  $\mathcal{F}_{\partial B(x^l, a^l)}$ is an isometry on
$(L^2(\partial B(x^l, a^l)))^3$.  Theorem \ref{pi_norm_sphere} then implies that $\| \Pi_{\Gamma} \|\leq \frac{1}{2}$.
\end{proof}


We now turn to study the second integral on the right hand side of equation \eqref{bie_p},
which we rewrite
making the change of variables $\zeta=\xi-\nu$, as
\begin{equation}
\label{P*nu_1}
\nabla_\xi \cdot  \int_{\partial \Omega_F}  n(\nu) \mathcal{P}(\xi-\nu) \cdot u (\nu) \, dS_\nu
= \nabla_\xi \cdot  \int_{ \xi-\partial \Omega_F}  n(\xi-\zeta) \mathcal{P}(\zeta) \cdot u (\xi-\zeta) \, dS_\zeta,
\end{equation}
where $\xi-\partial \Omega_F : = \{ \zeta = \xi-\nu \ ; \ \nu \in \partial \Omega_F\}$.  Recall that this integral must be interpreted in a principal value sense.
We temporarily assume   $u \in C^\infty( \partial \Omega_F)$ and note that  in the sense of distributions $\nabla \Pi u = \Pi \nabla u$.
Consequently, from the definition of distributional derivative:
\begin{align}
\nabla_\xi \cdot &  \int_{ \xi-\partial \Omega_F}  n(\xi-\zeta) \mathcal{P}(\zeta) \cdot u (\xi-\zeta) dS_\zeta \nonumber \\
=& \int_{\xi-\partial \Omega_F} \left[ \nabla_\xi \cdot n(\xi-\zeta) \mathcal{P}(\zeta)\cdot u (\xi-\zeta)
+ n(\xi-\zeta) \cdot \left(\nabla_\xi u (\xi-\zeta) \mathcal{P}(\zeta) \right)\right] dS_\zeta \nonumber \\
=&- \int_{\xi-\partial \Omega_F} \left[ \nabla_\zeta \cdot n(\xi-\zeta) \mathcal{P}(\zeta) \cdot u (\xi-\zeta)
+ n(\xi-\zeta) \cdot \left(\nabla_\zeta  u (\xi-\zeta) \mathcal{P}(\zeta) \right)\right] dS_\zeta \nonumber \\
=&\int_{\partial \Omega_F} \left[ \nabla_\nu \cdot n(\nu) \mathcal{P}(\xi-\nu) \cdot u (\nu)
+ n(\nu) \cdot \left( \nabla_\nu u(\nu) \mathcal{P}(\xi-\nu)\right)\right] dS_\nu, \nonumber
\end{align}
where the last equality follows by setting $\nu=\xi-\zeta$. Since $\mathcal{P}$ is a Fourier multiplier in $L^2$, this last expression is valid by density for any $u\in H^1(\partial \Omega_F)$.
Furthermore,  for $x\in\partial B(x^l,a)$, $\nabla \cdot n = \frac{3}{a}$ (or $-\frac{3}{a}$ if $n$ is the inward-facing normal). Thus, for $\Omega_F:= B(0,R) \setminus \cup_{l=1}^N B(x^l,a)$,
\begin{equation*}
\nabla_{\xi} \cdot \int_{\partial \Omega_F} n  \mathcal{P}(\xi-\nu) \cdot u(\nu) dS_\nu = -\sum_{l=1}^N \frac{3}{a}\Pi_{\partial B(x^l,a)} u + \frac{3}{R}\Pi_{\partial B(0,R)} u +
\Pi_{\partial \Omega_F} \frac{\partial u}{\partial n}.
\end{equation*}

Under the hypotheses of Theorem \ref{final_estimate_theorem}, Theorem \ref{Pi_estimate_thm} then gives
\begin{equation}
\label{second_integral_estimate}
\bigg\| \nabla_{\xi} \cdot \int_{\partial \Omega_F} n \mathcal{P}(\xi-\nu) \cdot u(\nu) dS_\nu \bigg\|_{\partial \Omega_F}
\leq \frac{3}{2a} \left \| u \right \|_{\cup_l \partial B^l}
+ \frac{3}{2R} \left \| u \right \|_{\partial B(0,R)}
+ \frac{1}{2} \left \| \frac{\partial u}{\partial n} \right \|_{\partial \Omega_F}.
\end{equation}
Finally, combining equations \eqref{bie_p}, \eqref{second_integral_estimate}, and Theorem \ref{Pi_estimate_thm}, we obtain
\begin{equation*}
\| p \|_{\partial \Omega_F} \leq \frac{1}{2} \| t \|_{\partial \Omega_F} + \frac{3\eta}{a} \left \| u \right \|_{\cup_l \partial B^l}
+ \frac{3\eta}{R} \left \| u \right \|_{\partial \Omega}+ \eta\left\| \frac{\partial u}{\partial n} \right\|_{\partial \Omega_F}.
\end{equation*}

The results in this section can be extended to other particle shapes by using mappings from their surfaces to $S^2$.

\section{Error estimate}
\label{sec_error_estimate}

With estimate \eqref{final_estimate_traction} on the boundary tractions in place, we can now turn to establishing the needed error bounds for the effective viscosity.
As in Section \ref{sec_setup}, we model a suspension of rigid spheres of radius $a$,  centered on the lattice $\mathbb{Z}^3 \cap B(0,R-1)$, that is, $d=1$. 
The limit $\phi:= \frac{4\pi}{3} a^3 \rightarrow 0$ then corresponds to $a\rightarrow 0$ for a fixed, but arbitrary, number of spheres $N$.
Recall that the effective viscosity is defined by
\begin{equation}
\label{ev_energy_rep}
\hat{\eta} = (2 | \Omega | \epsilon : \epsilon)^{-1} E(u),
\end{equation}
where $E(u ):= 2 \eta \int_{\Omega_F} e : e \, dx$,
$e = \frac{1}{2} \left( \nabla u + \nabla u^T\right)$,
and $\epsilon$ is the constant, symmetric, trace free matrix describing the boundary conditions in equation \eqref{suspension_pde}.
We will bound $\hat{\eta}$ below and above by quantities which have the same asymptotic behavior to order $\phi^{3/2}$.

\begin{theorem}\label{thm_ev_arb_n}
There exists $C>0$ independent of $a$, $N$, and hence of $\phi:=\frac{4\pi}{3} a^3$, such that $\forall 0<\phi<\frac{\pi}{6}$ the effective viscosity $\hat{\eta}$ satisfies
\begin{equation*}
\left|\hat{\eta}-\eta\left(1+\frac{5}{2} \phi\right) \right| \leq C \eta \phi^{3/2}.
\end{equation*}
\end{theorem}

\begin{remark}
The upper bound for $\phi$ comes from the condition that the spherical inclusions do not overlap.
\end{remark}

\begin{proof}
The proof of Theorem \ref{thm_ev_arb_n}
is an immediate consequence of Lemmas \ref{lower_bound_lemma}, \ref{upper_bound_lemma}, and \ref{lemma_ev_upper_bound}, given below.
\end{proof}

\subsection{Lower bound for the effective viscosity}

In order to obtain a lower bound for $\hat{\eta}$, we will use a similar technique to that used in \cite{keller1967}.  As in \cite{keller1967}, we will bound $E(u)$ from below, but we will use the dilute approximation to $u$ (the solution of equation \eqref{suspension_pde}) given by
\begin{equation}
\label{u^-}
\begin{split}
u^-(x)&:=\epsilon x + \sum_{l=1}^N \left[u^1(x-x^l) -\epsilon x\right],
\end{split}
\end{equation}
where $u^1$ solves
\begin{equation}
\label{dilute_problem}
\left\{
\begin{array}{lr}
\eta \Delta u^1 = \nabla p^1 \text{, } \nabla \cdot u^1 = 0 & x\in \R^3 \setminus B(0,a) \\
u^1 = 0 & x \in \partial B(0,a) \\
u^1 \rightarrow \epsilon x & x\rightarrow \infty \\
\end{array}
\right.
\end{equation}
and is given by $u^1=\epsilon x + u^{1\prime}$, with
\begin{equation}
\label{u^1prime}
u^{1\prime}(x) =- \epsilon x \frac{a^5}{|x|^5}- x ( \epsilon:  x x ) \frac{5}{2} \left(\frac{a^3}{|x|^5}-\frac{a^5}{|x|^7}\right).
\end{equation}
The corresponding pressure is given by
\begin{equation}
\label{p^1prime}
p^1(x)=p^{1\prime}(x) = -5a^3\eta \epsilon : \frac{x x}{|x|^5}.
\end{equation}

In order to obtain error estimates independent of the number of inclusions $N$, we will need to regularize various sums.  We begin with two lemmas.

\begin{lemma}
\label{sum_converges_lemma}
For all $x\in B\left(0,\frac{\sqrt{3}}{2}\right)$ and all $0 \leq \rho < \infty$,
\begin{equation}
\label{sum_2}
\left| \sum_{z\in \mathbb{Z}^3\setminus \{0\} \cap B(0,\rho)} \left( \frac{1}{|x-z|^2} -\frac{1}{|z|^2} \right) \right| \leq (876 + 504 \sqrt{3})\pi.
\end{equation}
\end{lemma}


\begin{lemma}
\label{sum_converges_lemma_3}
For all $x\in B(0,\frac{\sqrt{3}}{2})$, the sum
\begin{equation}
\label{sum_3}
\left|\sum_{z\in \mathbb{Z}^3\setminus \{0\}} \left( \frac{1}{|x-z|^3} -\frac{1}{|z|^3} \right) \right| \leq 2\pi \left(159 + 92 \sqrt{3} + \log 4 - \log (7-4\sqrt{3}) \right).
\end{equation}
\end{lemma}


Both estimates follow by replacing the sums with integrals and using spherical coordinates to bound the integrals.  The convergence of the integrals was checked by hand and the explicit numerical bounds were computed by Mathematica.

Next, we show below that we can use $u^-$, defined in equation \eqref{u^-}, to obtain a lower bound for $\hat{\eta}$.
\begin{lemma}
\label{lower_bound_lemma}
For $\hat{\eta}$ defined in equation \eqref{ev_energy_rep}, there exists $C>0$ independent of $\phi$ and $N$ such that \  $\forall \phi < \frac{\pi}{6}$,
\quad $
\eta \left(1+\frac{5}{2} \phi \right) \leq \hat{\eta} + C \eta \phi^{5/3}.
$
\end{lemma}

\begin{proof}
The proof is a calculation.  First, note that
\begin{equation}
\label{lower_bound_lemma_eq_1}
E(u-u^-) = E(u)+E(u^-)-4 \eta \int_{\Omega_F}  e^-:e \, dx.
\end{equation}
Now, integrating by parts and taking advantage of the boundary conditions on $u$, we have
\begin{align}
2\eta \int_{\Omega_F} & e^-: e \, dx = \int_{\Omega_F} \sigma^-: \nabla u \, dx  = \int_{\partial \Omega} \sigma^- n \cdot  \epsilon x dS  \nonumber \\
&+ \sum_{l=1}^N \int_{\partial B^l} \sigma^- n \cdot \left(v^l+ \omega^l \times (x-x^l)\right) dS
= \int_{\partial \Omega}\sigma^- n \cdot  \epsilon x  dS, \label{lower_bound_lemma_eq_2}
\end{align}
since the particles are force and torque-free (that is, $\int_{\partial B^l} \sigma^- n \, dS = 0$ and $\int_{\partial B^l} \sigma^- n \times \left(x-x^l\right) dS=0$).
Now,
\begin{align}
E(u^-)&-2\int_{\partial \Omega} \sigma^- n  \cdot \epsilon x dS
= \int_{\Omega_F} \sigma^- : e^- \, dx -2\int_{\partial \Omega} \sigma^- n \cdot  \epsilon x dS \nonumber \\
=& -\int_{\partial \Omega} \sigma^- n \cdot  \epsilon x  dS + \int_{\partial \Omega}\sigma^- n  \cdot u^{-\prime}dS + \sum_l \int_{\partial B^l} \sigma^- n \cdot u^{-} dS \nonumber \\
=:& I_1+I_2+I_3. \label{lower_bound_lemma_eq_3}
\end{align}
First, we can write $\sigma^-=2\eta \epsilon + \sigma^{-\prime}$ and integrate $I_1$ by parts to obtain
\begin{align}
I_1 
=& -\int_{\Omega_F} \sigma^{-\prime} : \epsilon \, dx + \sum_l \int_{\partial B^l} \sigma^{-\prime} n \cdot  \epsilon x dS
- 2 \eta \int_{\partial \Omega} \epsilon n \cdot  \epsilon x dS \nonumber \\
=& -2 \eta \epsilon : \int_{\Omega_F} e^{-\prime} \, dx + \sum_l \int_{\partial B^l} \sigma^{-\prime} n \cdot  \epsilon x dS - \frac{8\pi R^3 \eta}{3} \epsilon : \epsilon\nonumber  \\
=& -2\eta \epsilon :  \int_{\partial \Omega} u^{-\prime} n \, dS - 2 \eta \sum_l \epsilon :\int_{\partial B^l}  u^{-\prime} n \, dS\nonumber
\\ & \qquad \qquad  \qquad
+ \sum_l \int_{\partial B^l} \sigma^{-\prime} n \cdot  \epsilon x dS  - \frac{8\pi R^3 \eta}{3} \epsilon: \epsilon.\nonumber
\end{align}
Some of these integrals above can be evaluated explicitly (see e.g. \cite{batchelor}) to give
\begin{equation}
\label{lower_bound_lemma_I1}
\begin{split}
I_1 =& - \frac{8\pi R^3 \eta}{3} \epsilon : \epsilon \left(1+\frac{5}{2} \phi \right) - 2 \eta \int_{\partial \Omega} \epsilon : u^{-\prime} n \, dS \\
&+ \sum_l \sum_{m \neq l} \epsilon : \int_{\partial B^l} \left[ \sigma^{1\prime} (x-x^m) n x - 2 \eta u^{1\prime} (x-x^m) n \right] dS
\end{split}
\end{equation}
Now, by integrating by parts and recalling that $n$ is the inward-facing normal to $\partial B^l$, we have
\begin{align}
\sum_{m \neq l}& \epsilon :\int_{\partial B^l} \sigma^{1\prime}  (x-x^m) n x dS
= -\sum_{m \neq l} \epsilon :\int_{B^l} \sigma^{1\prime} (x-x^m) \, dx= \nonumber\\
  &-2 \eta\sum_{m \neq l} \epsilon :\int_{B^l} e^{1\prime} (x-x^m) \, dx
= 2\eta\sum_{m \neq l} \epsilon :\int_{\partial B^l} u^{1\prime} (x-x^m) n \, dS. \label{lower_bound_lemma_I1_term}
\end{align}
Combining equations \eqref{lower_bound_lemma_I1} and \eqref{lower_bound_lemma_I1_term} gives
\begin{equation}
\label{lower_bound_lemma_I1_final}
I_1 = - \frac{8\pi R^3 \eta}{3} \epsilon : \epsilon \left(1+\frac{5}{2} \phi \right) - 2 \eta \epsilon :\int_{\partial \Omega} u^{-\prime} n \, dS.
\end{equation}
We can split $I_2$ into two integrals by writing  again $\sigma^-=2\eta \epsilon+ \sigma^{-\prime}$:
\begin{equation}
\label{lower_bound_lemma_I2}
I_2 = \int_{\partial \Omega} \sigma^- : u^{-\prime} n \, dS
= 2\eta \epsilon: \int_{\partial \Omega}  u^{-\prime} n \, dS + \int_{\partial \Omega} \sigma^{-\prime}: u^{-\prime} n \, dS.
\end{equation}
The first term above will cancel the second term in equation \eqref{lower_bound_lemma_I1_final}.
Now, we can estimate
\begin{align}
\left| \int_{\partial \Omega} \right. & \left. \sigma^{-\prime} : u^{-\prime}
 n \, dS \right|
\label{lower_bound_lemma_I2_estimate_2}\\
= & \left| \int_{\partial \Omega} \sum_l \sigma^{1\prime} (x-x^l):
\sum_m u^{1\prime}(x-x^m) n \, dS \right|, \nonumber \\
\intertext{which, using the formulae for $u^{1\prime}$ and $p^{1\prime}$, given in equations \eqref{u^1prime} and \eqref{p^1prime}, respectively, and the fact that $\int_{\partial \Omega} \sigma^{-\prime} n \, dS=0$, is}
\leq & C \left|\partial \Omega\right| \left\| \sum_l \sigma^{1\prime}(x-x^l) n \right\|_{(L^\infty(\partial \Omega))^3} \nonumber \\
& \cdot a^3 \inf_{C \in \mathbb{R}} \left\| \sum_{z\in \mathbb{Z}^3\setminus \{ 0\} \cap B(0,2N^{1/3})} \frac{1}{|x-z|^2} -  C \right\|_{L^\infty\left( B\left(\frac{\sqrt{3}}{2}\right)\right)} \nonumber \\
\leq & C a^3 \left|\partial \Omega\right| \left\| \sum_l \sigma^{1\prime}(x-x^l) n \right\|_{(L^\infty(\partial \Omega))^3} \nonumber \\
& \cdot \left\| \sum_{z\in \mathbb{Z}^3\setminus \{ 0\} \cap B(0,2N^{1/3})} \left(\frac{1}{|x-z|^2} -  \frac{1}{|z|^2} \right)\right\|_{L^\infty\left( B\left(\frac{\sqrt{3}}{2}\right)\right)}. \nonumber \\
\intertext{Applying Lemma \ref{sum_converges_lemma} and recalling that $|\Omega| \propto N$, this is}
\leq & C N^{2/3} \left[ \sum_{z\in \mathbb{Z}^3\setminus \{ 0\}} \frac{a^3}{|z|^3} \right]
a^3
\leq C a^6 N^{2/3} \log N. \nonumber
\end{align}
We  estimate $I_3$ in a similar fashion. First, we note that $u^1(x-x^l)=u^-(x) - \sum_{m \neq l} u^{1\prime}(x-x^m)$ is constant on $\partial B^l$.
Hence,
\begin{align}
I_3 =& \sum_l \int_{\partial B^l} \sigma^- : u^{-} n \, dS
=\sum_l \int_{\partial B^l} \sigma^-: \sum_{m \neq l} u^{1\prime}(x-x^m) n \, dS \nonumber
\end{align}
To regularize the sum above, we set for $x\in \mathbb{Z}^3$,
\begin{equation}
\label{sigma_reg_term_def}
\sigma^{r\prime} (x) := \left\{
\begin{array}{lr}
-p^{1\prime}(x) I + \eta \left( \nabla u^{r\prime}(x) + (\nabla u^{r\prime}(x))^T\right) & x \in \mathbb{Z}^3 \setminus \{0\} \\
0 & x = 0.
\end{array}
\right.
\end{equation}
We divide the integral further by writing $\sigma^-=2 \eta E + \sigma^{-\prime}$:
\begin{align}
I_3 =& \sum_l \int_{\partial B^l} \sigma^- : \sum_{m \neq l} u^{1\prime}(x-x^m) n \, dS \nonumber \\
=& 2 \eta  \sum_l E : \int_{\partial B^l}  \sum_{m \neq l} u^{1\prime}(x-x^m) n \, dS
+ \sum_l \int_{\partial B^l} \sigma^{-\prime} : \sum_{m \neq l} u^{1\prime}(x-x^m) n \, dS \nonumber \\
=:& I_3^a+I_3^b. \label{lower_bound_lemma_I4}
\end{align}
Given  Lemma \ref{sum_converges_lemma} and the formulae for $u^{1\prime}$ and $p^{1\prime}$, given in \eqref{u^1prime} and \eqref{p^1prime}, respectively, we can estimate $I_3^a$  in a manner similar to equation \eqref{lower_bound_lemma_I2_estimate_2}, taking advantage of the fact that $\int_{\partial B^l} n \, dS = 0$:
\begin{equation}
\label{lower_bound_lemma_I4a}
\begin{split}
\left| I_3^a \right| =& \left| 2 \eta  \sum_l E : \int_{\partial B^l}  \sum_{m \neq l} u^{1\prime}(x-x^m)  n \, dS \right| \\ 
\leq & C \eta a^3 \sum_l \left| \partial B^l \right| \inf_{C \in \mathbb{R}}\left\| \sum_{z \in \mathbb{Z}^3 \setminus \{0\} \cap B(0,2N^{1/3})} \frac{1}{|x-z|^2} - C \right\|_{L^\infty(\partial B(0,a))} \\
\leq & C \eta a^3 \sum_l \left| \partial B^l \right| \left\| \sum_{z \in \mathbb{Z}^3 \setminus \{0\} \cap B(0,2N^{1/3})}  \left( \frac{1}{|x-z|^2} - \frac{1}{|z|^2} \right) \right\|_{L^\infty(\partial B(0,a))} \\
\leq & C \eta N a^5.
\end{split}
\end{equation}
To estimate $I_3^b$, note that, by the Cauchy-Schwarz inequality and the fact that $\int_{\partial B^l} \sigma^{-\prime} n \, dS=0$,
\begin{equation*}
\begin{split}
\left| I_3^b \right|
\leq& \left| \sum_l \int_{\partial B^l} \sigma^{-\prime} : \sum_{m \neq l} u^{1\prime}(x-x^m) n \, dS \right| \\
\leq & C a^3 \left\| \sigma^{-\prime} n \right\|_{\cup_l \partial B^l} N^{1/2} \inf_{C\in \mathbb{R}} \left\| \sum_{z \in \mathbb{Z}^3 \setminus \{0\} \cap B(0,2N^{1/3})} \frac{1}{|x-z|^2} - C \right\|_{\partial B(0,a)} \\
\leq & C a^3 \left\| \sigma^{-\prime} n \right\|_{\cup_l \partial B^l} N^{1/2} \left\| \sum_{z \in \mathbb{Z}^3 \setminus \{0\} \cap B(0,2N^{1/3})} \left( \frac{1}{|x-z|^2} - \frac{1}{|z|^2} \right) \right\|_{\partial B(0,a)} \\
\leq & C N^{1/2} a^4 \left\| \sigma^{-\prime} n \right\|_{\cup_l \partial B^l}.
\end{split}
\end{equation*}
Proceeding as in equation \eqref{lower_bound_lemma_I2_estimate_2},
since $\int_{\partial B^l} \sigma^{1\prime} (x-x^k) n \, dS = 0$,
\begin{align}
\bigg\| \sum_{k} &\sigma^{1\prime} (x-x^k) n \bigg\|_{\cup_l \partial B^l} \leq
\bigg\| \sum_{k}   \left[ \sigma^{1\prime} (x-x^k) - \sigma^{r\prime} (x^k) \right] n \bigg\|_{\cup_l \partial B^l}. \nonumber \\
\intertext{Making use of the formulae for $u^{1\prime}$ and $p^{1\prime}$, given in equations \eqref{u^1prime} and \eqref{p^1prime}, respectively, and Lemma \ref{sum_converges_lemma_3}, we get that this is}
\leq & C N^{1/2} a  \left\| \sum_k \left[\sigma^{1\prime} (x-x^k) - \sigma^{r\prime} (x^k) \right] n \right\|_{(L^\infty(\cup_l \partial B^l))^3} \nonumber \\
\leq & C \eta N^{1/2} a \left\| a^3 \sum_{z \in \mathbb{Z}^3 \setminus \{0 \}} \left( \frac{1}{|x-z|^3} - \frac{1}{|z|^3} \right) \right\|_{L^\infty(\partial B(0,a))} 
\leq  C \eta N^{1/2} a^4. \label{lower_bound_lemma_I4b_4}
\end{align}
The bounds (\ref{lower_bound_lemma_I4}--\ref{lower_bound_lemma_I4b_4}) imply that
\begin{equation}
\label{lower_bound_lemma_I4_final}
|I_3| \leq C \eta N a^5.
\end{equation}
We recall that $|\Omega| = 4/3\, \pi \, R^3 \propto N$.
Then, \eqref{ev_energy_rep}  and   previous estimates combined finally yield:
\begin{equation*}
0\leq \frac {E(u-u^-)}{2|\Omega| \epsilon : \epsilon} \leq \hat{\eta} - \eta  \left(1+\frac{5}{2}\phi\right) + C
\eta  \phi^{5/3},
\end{equation*}
for some constant $C>0$ dependent on $u$, but  independent of $N$ and $\phi$.
\end{proof}

\subsection{Upper bound for the effective viscosity}

We will next show that $u^+$,  solution of the boundary value problem
\begin{equation}
\label{u^+}
\left\{
\begin{array}{lr}
\eta \Delta u^+ = \nabla p^+ \text{, } \nabla \cdot u^+ = 0 & x \in \Omega_F \\
u^+ = \epsilon x^l & x \in \partial B^l \\
u^+ = \epsilon x & x \in \partial \Omega,
\end{array}
\right.
\end{equation}
provides an upper bound for the effective viscosity.

\begin{lemma}
\label{upper_bound_lemma}
Let $\hat{\eta}$ be defined in
equation \eqref{definition_of_ev} and let
\begin{equation}
\label{def_ev_u+}
\hat{\eta}(u^+):= (2|\Omega| \epsilon:\epsilon)^{-1} E(u^+).
\end{equation}
Then, \
$
\hat{\eta} \leq \hat{\eta}(u^+).
$
\end{lemma}

\begin{proof}
We begin by writing
\begin{equation*}
E(u^+-u) = E(u^+) + E(u) - 2  \eta \int_{\Omega_F} e^+ : e \, dx,
\end{equation*}
Next we observe that, since the particles are force and torque--free,
\[
\begin{aligned}
E(u) &= 2 \eta\int_{\Omega_F} e : e \, dx = \int_{\Omega_F} \sigma : e \, dx =  \int_{\Omega_F} \sigma: \nabla u \, dx \nonumber \\
&= \int_{\partial \Omega} \sigma n \cdot \epsilon x \, dS
+ \sum_l \int_{\partial B^l} \sigma n \cdot \left(v^l + \omega^l \times (x - x^l) \right) dS  \\
&= \int_{\partial \Omega} \sigma n \cdot  \epsilon x \, dS
  =  \int_{\Omega_F} e^+ : \sigma \, dx = 2 \eta \int_{\Omega_F} e^+: e \, dx.
\end{aligned}
\end{equation*}
Therefore, $E(u^+-u)=E(u^+)-E(u) \geq 0$, that is,  $E(u) \leq E(u^+)$, which by definition implies $\hat{\eta} \leq \hat{\eta}(u^+).$
\end{proof}

In light of Lemmas \ref{lower_bound_lemma} and \ref{upper_bound_lemma}, it remains to prove that there exists a $C>0$ independent of $\phi,N$ such that for all $\phi < \frac{\pi}{6}$,
\begin{equation}
\label{u^+_asymptotics}
\left| \hat{\eta}(u^+) - \eta \left(1+\frac{5}{2} \phi \right) \right| \leq C \eta \phi^{3/2}.
\end{equation}
To do so, we introduce a dilute approximation for $u^+$:

\begin{equation}
\label{new_u^d}
\begin{split}
u^d&:=E x + \sum_l \left[u^1(x-x^l)
-E x\right] - C^d,
\end{split}
\end{equation}
where $u^1$ solves equation \eqref{dilute_problem}.  The constant $C^d$, which will be determined later, is added to regularize the sum--that is, to ensure $u^d$ remains finite as $N\rightarrow \infty$.

We begin with a preliminary lemma.

\begin{lemma}
\label{lemma_ev_representation_arb_n}
There exists $C>0$ independent of $a$ and $N$ such that
\begin{equation}
\label{lemma_formula_for_ev_arb_n}
\left| \hat{\eta}(u^+) - \eta\left(1 +  \frac{5}{2} \phi \right)\right| \leq \left|
\frac{1}{2\epsilon : \epsilon |\Omega |} \sum_l \epsilon: \int_{\partial B^l} \sigma^E n \left( x - x^l\right) dS \right| + C  \eta\phi^{5/3},
\end{equation}
where $\sigma^E$ is the stress corresponding to the flow $u^E:=u^+-u^d$, where $u^+$ is the solution to equation \eqref{u^+} and $u^d$ is defined in equation \eqref{new_u^d}
\end{lemma}

\begin{proof}
We integrate the right-hand side of equation \eqref{def_ev_u+} by parts, utilizing  \eqref{u^+},
\begin{equation}
\label{lemma_formula_for_ev_1}
\hat{\eta}(u^+) 
= \frac{1}{2 |\Omega| \epsilon : \epsilon} \int_{\Omega_F} \sigma^+ : e^+ \, dx 
= \frac{\epsilon:}{2 |\Omega| \epsilon : \epsilon} \left[ \int_{\partial \Omega} \sigma^+ n x dS + \sum_l \int_{\partial B^l} \sigma n x^l dS \right].
\end{equation}
where by further integration by parts we can write (recall $\sigma^+ = -p^+ I + 2\eta e^+$):
\begin{align}
\epsilon : \int_{\partial \Omega} \sigma^+ n x dS =&
2 \eta \epsilon: \int_{\Omega_F} e^+ \, dx- \sum_l \epsilon : \int_{\partial B^l} \sigma^+ n x dS \nonumber \\
=& 2 \eta \epsilon : \int_{\partial \Omega} \epsilon x n \, dS 
+ \epsilon : \sum_l \int_{\partial B^l} \left(2 \eta \epsilon n x^l - \sigma^+ n x \right)dS \nonumber \\
=& 2 \eta |\Omega| \epsilon : \epsilon - \epsilon : \sum_l \int_{\partial B^l} \sigma^+ n x dS. \label{lemma_formula_for_ev_2}
\end{align}
We then have, reversing the orientation so that $n$  is the outward-facing normal to $\partial B^l$,
\begin{equation*}
\hat{\eta}(u^+) =\eta + \frac{1}{2 |\Omega| \epsilon: \epsilon} \sum_l  \epsilon: \int_{\partial B^l} \sigma^+ n (x - x^l) dS.
\end{equation*}
Substituting $\sigma^{+}=2 \eta \epsilon + \sigma^{d\prime} + \sigma^E$ in the above equation, and integrating  explicitly some terms (see, for example, \cite{batchelor}) gives
\begin{align}
\hat{\eta}(u^+) = &\eta\left(1 +  \frac{5}{2}\phi\right)+
\frac{1}{2|\Omega|\epsilon: \epsilon} \sum_l \epsilon : \int_{\partial B^l} \sigma^E n (x-x^l) dS \nonumber \\
& +\frac{1}{2|\Omega|\epsilon: \epsilon} \sum_l \epsilon : \int_{\partial B^l} \sum_{m \neq l}\sigma^{1\prime}(x-x^l) n \left(x-x^l\right) dS. \nonumber
\end{align}
We denote the last integral above by $I$, and we observe that, since $u^1$ solves  \eqref{dilute_problem}, integrating by parts,
\begin{equation*}
\begin{split}
I
=& -  \epsilon :\int_{B^l} \sum_{m \neq l} \left( \nabla \cdot \sigma^{1\prime} x + \sigma^{1\prime}(x-x^m) \right) \, dx \\
=& -2 \eta   \epsilon :\int_{B^l} \sum_{m \neq l} e^{1\prime}(x-x^m) \, dx
= -2 \eta   \epsilon: \int_{\partial B^l} \sum_{m \neq l} u^{1\prime}(x-x^m) n \, dS
\end{split}
\end{equation*}

Note that, since $\int_{\partial B^l} n \, dS = 0$, by Lemma \ref{sum_converges_lemma} and formula  \eqref{u^1prime} for $u^{1\prime}$, we have
\begin{equation*}
\begin{split}
\bigg|\int_{\partial B^l}  & \sum_{m \neq l}u^{1\prime}(x-x^m)  n \, dS \bigg| \\
& \leq Ca^3\left| \partial B^l \right| \inf_{C\in\mathbb{R}} \left\| \sum_{z \in \mathbb{Z}^3 \setminus \{0\} \cap B(0, 2 N^{1/3})} \frac{1}{|x-z|^2} - C \right\|_{(L^\infty(\partial B(0,a)))^3} \\
& \leq Ca^3\left| \partial B^l \right| \left\| \sum_{z \in \mathbb{Z}^3 \setminus \{0\} \cap B(0, 2 N^{1/3})} \left( \frac{1}{|x-z|^2} - \frac{1}{|z|^2} \right) \right\|_{(L^\infty(\partial B(0,a)))^3} \\
& \leq C a^5.
\end{split}
\end{equation*}
Therefore, recalling that $|\Omega | \propto N$,
\begin{equation*}
\left| \frac{1}{2|\Omega|\epsilon: \epsilon} \sum_l  \epsilon : \int_{\partial B^l} \sum_{m \neq l} \sigma^{1\prime}(x-x^m) n \left(x-x^l\right) dS \right| \leq C \eta \phi^{5/3}.
\end{equation*}
\end{proof}

The next Lemma is a straightforward consequence of Lemma
\ref{lemma_ev_representation_arb_n} and Corollary \ref{final_t_estimate}.

\begin{lemma}
\label{error_estimate_lemma_2_arb_n}
There exists a $C>0$ independent of $\phi$ and $N$ such that for all $\phi < \frac{\pi}{6}$,
\begin{equation}
\label{error_estimate_u^E_arb_n}
\begin{split}
\left| \hat{\eta}(u^+) - \eta \left(1+\frac{5}{2} \phi \right) \right| 
\leq & C \eta \frac{\phi^{1/3}}{N} \left\| u^E \right\|_{\cup_l \partial B^l}
+ C \eta \frac{\phi^{1/2}}{N^{1/2}} \left\| u^E \right\|_{\partial \Omega_F} \\
& + C \eta \frac{\phi^{2/3}}{N} \sum_{i=1}^2 \left\| \frac{\partial u^E}{\partial \tau^i} \right\|_{\partial \Omega_F} + C \eta \phi^{5/3}.
\end{split}
\end{equation}
\end{lemma}

\begin{proof}
From Lemma \ref{lemma_ev_representation_arb_n}, we can write
\begin{equation*}
\begin{aligned}
\left| \hat{\eta}(u^+)-\eta\left(1+  \frac{5}{2} \phi \right) \right| &\leq
\left| \frac{1}{2\epsilon : \epsilon |\Omega |} \sum_l \epsilon: \int_{\partial B^l} \sigma^E n (x - x^l) dS \right|+ C \eta \phi^{5/3} \\
&\leq  C \left(\frac{a}{N} \sum_l \int_{\partial B^l} | \sigma^E  n | dS +  \eta \phi^{5/3} \right),
\end{aligned}
\end{equation*}
where the second inequality follows from the  Cauchy-Schwarz inequality,
 noting that  $|x - x^l|=a$ for $x\in\partial B^l$ and  $|\Omega | \propto N$. \ Then  Corollary \ref{final_t_estimate} implies that
\begin{align}
\left| \hat{\eta}(u^+)  \right . & \left .-\eta\left(1+  \frac{5}{2} \phi \right) \right|
 \leq C \left(\frac{a^2}{N} \left\| \sigma^E n \right\|_{\partial \Omega_F} +  \eta \phi^{5/3}  \right) \nonumber \\
\leq & C \left( \eta \frac{a}{N}\left\| u^E \right\|_{\cup_l \partial B^l}
\eta \frac{a^{3/2}}{N^{1/2}}\left\| u^E \right\|_{\partial \Omega_F}
+  \eta \frac{a^2}{N} \sum_{i=1}^2 \left\| \frac{\partial u^E}{\partial \tau^i} \right\|_{\partial \Omega_F} + \eta \phi^{5/3} \right), \nonumber
\end{align}
which, upon making the substitution $a=\left(\frac{3\phi}{4\pi}\right)^{1/3}$, is precisely equation \eqref{error_estimate_u^E_arb_n}.

\end{proof}

The upper bound \eqref{u^+_asymptotics} will now follow from the previous lemma once we  estimate $\left\| u^E \right\|_{H^1(\partial \Omega_F)}$.  For this purpose, it is useful to write a PDE that $u^E$ satisfies:
\begin{equation}
\label{error_pde_arb_n}
\left\{
\begin{array}{lr}
\eta \Delta u^E = \nabla p^E \text{, } \nabla \cdot u^E = 0 & x\in\Omega_F \\
u^E = -\Theta^l +  & x \in \partial B^l \text{, } x^l \neq 0 \\
u^E = -\Theta^l + u^{1\prime}(x) & x \in \partial B^1 \\
u^E = -\Theta^{\partial\Omega} & x \in \partial \Omega, \\
\end{array}
\right.
\end{equation}
where
$\Theta^l = \sum_{k\neq l} u^{1\prime}(x-x^k)-C^d$ and
$\Theta^{\partial\Omega} = \sum_k u^{1\prime}(x-x^k)-C^d$.
We will use this PDE to prove the main result of this section.

\begin{lemma}
\label{lemma_ev_upper_bound}
There exists $C>0$ independent of $a$, $N$, and hence $\phi$,
such that, for all $0<\phi<\frac{\pi}{6}$, equation \eqref{u^+_asymptotics} holds.

\end{lemma}

\begin{proof}
From Lemma \ref{error_estimate_lemma_2_arb_n}, we have that there exists a $C>0$ independent of $N,a$, and hence $\phi$ such that
\begin{align}
\left| \hat{\eta}(u^+) - \eta \left(1+\frac{5}{2} \phi \right) \right| 
\leq & C \eta \frac{\phi^{1/3}}{N} \left\| u^E \right\|_{\cup_l \partial B^l}
+ C \eta \frac{\phi^{1/2}}{N^{1/2}} \left\| u^E \right\|_{\partial \Omega_F} \nonumber \\
& + C \eta \frac{\phi^{2/3}}{N} \sum_{i=1}^2 \left\| \frac{\partial u^E}{\partial \tau^i} \right\|_{\partial \Omega_F}+C\eta \phi^{5/3}. \label{lemma_ev_upper_bound_1}
\end{align}

Now, using the formula for $u^{1\prime}$ in equation \eqref{u^1prime}, selecting $C^d$ such that the infimum
\begin{equation*}
\inf_{C^d \in \mathbb{R}^3} \left\| u^E \right \|_{(L^\infty(\partial \Omega_F))^3}
\end{equation*}
is achieved
and applying Lemma \ref{sum_converges_lemma},
\begin{align}
\left\| u^E \right\|_{\cup_l \partial B^l} & \leq C N^{1/2} a \left\| u^E \right \|_{(L^\infty(\partial \Omega_F))^3} \nonumber \\
& \leq C N^{1/2} a \left\| a^3 \sum_{z \in \mathbb{Z}^3 \setminus \{0\} \cap B(0,2N^{1/3})} \left( \frac{1}{|x-z|^2} -\frac{1}{|z|^2} \right) \right\|_{L^\infty(\partial B(0,a))}
\nonumber \\ &
\leq C N^{1/2} \phi^{4/3}. \label{lemma_ev_upper_bound_2} \\
\intertext{Similarly,}
\left\| u^E \right\|_{\partial \Omega_F} &
\leq C N^{1/2} \phi, \label{lemma_ev_upper_bound_3}
\end{align}
\begin{equation}
\label{lemma_ev_upper_bound_4}
\begin{split}
\sum_{i=1}^2 \left\| \frac{\partial u^E}{\partial \tau^i} \right\|_{\partial \Omega_F}
& \leq C N^{1/2} \left\| \nabla u^E \right\|_{(L^\infty(\partial \Omega_F))^{3\times 3}} \\
& \leq C_1 N^{1/2} a^3 \sum_{l=1}^{C_2 N^{1/3}} \frac{1}{l^3} l^2 \leq C N^{1/2} \log N \phi. \\
\end{split}
\end{equation}
Combining equations (\ref{lemma_ev_upper_bound_1}--\ref{lemma_ev_upper_bound_4}), we have the desired result.

\end{proof}

\section*{Acknowledgements}
The authors would like to thank L.~Berlyand for suggesting the problem and D.~Karpeev for reading the manuscript and providing useful suggestions.  We would also like to thank Shawn Ryan for pointing out an error in one of the proofs.
The work of B.~Haines was supported by the DOE grant
DE-FG02-08ER25862 and NSF grant DMS-0708324.
A.~L.~Mazzucato was partially supported
by NSF grants DMS-0708902, DMS-1009713 and DMS-1009714.


\end{document}